\documentclass[reqno, 11pt]{amsart}
\usepackage[letterpaper, margin=1in]{geometry}
\usepackage[utf8]{inputenc} 
\usepackage[T1]{fontenc} 
\usepackage{amsmath}
\usepackage{amsthm} 
\usepackage{amssymb}
\usepackage{tikz}
\usepackage[colorlinks=true, allcolors=black]{hyperref}
\usepackage[nameinlink]{cleveref}

\theoremstyle{plain}
\newtheorem{theorem}{Theorem}[section]
\newtheorem{corollary}[theorem]{Corollary}
\newtheorem{lemma}[theorem]{Lemma}
\newtheorem{proposition}[theorem]{Proposition}

\newtheorem{definition}[theorem]{Definition}
\newtheorem{remark}[theorem]{Remark}

\newcommand{\abs}[1]{\left\lvert#1\right\rvert}
\newcommand{\norm}[1]{\left\lVert#1\right\rVert}
\newcommand{\sqb}[1]{\left[ #1 \right]}
\newcommand{\dd}{\mathrm{d}}
\newcommand{\II}[1]{\mathbf{1}{#1}}

\author{Saba Lepsveridze}
\address{Department of Mathematics, Massachusetts Institute of Technology}
\email{sabal@mit.edu}

\author{Oriol Solé-Pi}
\address{Department of Mathematics, Massachusetts Institute of Technology}
\email{oriolsp@mit.edu}

\begin{document}

\title{On the geodesic midrange crossing constant }

\begin{abstract}
We introduce an analytic framework for studying the number of crossings in planar straight-line and spherical geodesic drawings; we begin by analyzing drawings in a continuous setting and then pass to finite graphs. We use this approach to derive multiple results.

First, we show that the midrange crossing constant is strictly smaller than $8/(9\pi^2)$, improving the best known upper bound of Pach and Tóth from nearly thirty years ago. Second, we prove sharp lower bounds for some restricted drawing classes. For spherical geodesic drawings with uniformly distributed vertices, we determine the  minimum number of crossings at every edge density. We also prove lower bounds for planar straight line drawings in compact convex domains with uniform vertex sets. Consequently, we determine the restricted midrange crossing constant to be $8/(9\pi^2)$ in both of these settings. Finally, for two-layer bipartite drawings, with straight edges between two parallel lines, we determine the midrange crossing constant to be $2/3$.
\end{abstract}
 
\maketitle

\section{Introduction}\label{sec:intro}
A \textit{drawing} of a graph $G$ is a representation of the graph on either $\mathbb R^2$ or $\mathbb S^2$ in which the vertices are represented by distinct points and the edges by continuous simple curves connecting their respective endpoints. A drawing is called \textit{good} if no edge contains a vertex in its relative interior, no two edges are tangent to each other or intersect at infinitely many points, and no three edges share an interior point. A \textit{crossing} is a point which lies at the intersection of two edges but does not represent a vertex. The crossing number of $G$ is the least possible number of crossings across all good drawings of $G$, and we denote it by $\mathrm{cr}(G)$. Moreover, given a good drawing $\mathcal D$ of a graph, we write $\operatorname{cr}({\mathcal D})$ for the number of crossings induced by $\mathcal D$. By performing a stereographic projection, it can be seen that the definition of $\mathrm{cr}(G)$ does not depend on whether we consider drawings on $\mathbb R^2$ or $\mathbb S^2$.  We shall work exclusively with simple graphs. 

\subsection{The midrange crossing constant and our work} The celebrated Crossing Lemma, proved independently by Ajtai et al.~\cite{AjtaiChvatalNewbornSzemeredi1982} and Leighton~\cite{leighton1984new}, states that if $m\geq 4n$ then any graph $G$ with $n$ vertices and $m$ edges must satisfy \begin{equation}\label{eq:crossing_lemma}
\mathrm{cr}(G)\geq \frac{m^3}{64n^2}\,.
\end{equation} This bound, apart from being a foundational result in the study of crossing numbers, can be used to prove several other remarkable results, including the Szemerédi–Trotter Theorem~\cite{szekely1997crossing} from incidence geometry, and the currently best known upper bounds for the Erd\H{o}s unit distance problem~\cite{szekely1997crossing} and the number of $k$-edges produced by planar point sets~\cite{dey1998improved}. 

For positive integers $n$ and $m$, we write $\kappa(n,m)$ to denote the least crossing number among graphs with $n$ vertices and $m$ edges. In 1973, Erd\H{o}s and Guy~\cite{ErdosGuy1973} conjectured that along any sequence of pairs of the form $(n,m)$ with $n$ and $m/n$ both approaching infinity, the limit of $\kappa(n,m){n^2}/{m^3}$ exists (and this limit is independent of the sequence being used). As it turns out, this is not quite correct, but it becomes true after adding also the assumption that $n^2/m$ approaches infinity. 
\begin{theorem}[Pach, Spencer, Tóth~\cite{PachSpencerToth2000}]\label{thm:midrange_exists}
    Along every sequence of pairs $(n,m)$ with $n\ll m\ll n^2$, the limit \[\lim_{n\rightarrow \infty}\kappa(n,m)\frac{n^2}{m^3}=\kappa\] exists and is positive. 
\end{theorem}
An immediate consequence of this theorem is that the value of $\kappa$ does not depend on the precise sequence being considered. Pach, Spencer and Tóth called this constant $\kappa$ the \textit{midrange crossing constant}. The inequality~\eqref{eq:crossing_lemma} gives the lower bound $\kappa \geq 1/64$. There is a long and ongoing line of research~\cite{PachToth1997,pach2004improving,ackerman2019topological,bungener2024improving,bungener2026five} which aims at improving this lower bound. The best known result in this direction, due to B{\"u}ngener, Franz, Kaufmann and Pfister~\cite{bungener2026five}, tells us that $\kappa\geq\frac{648}{17689}\approx\frac{1}{27.3}$. On the other hand, the previous best known constructions suggested the value
\[
        \kappa=\frac{8}{9\pi^2}\,.
\]
Pach and T\'oth~\cite{PachToth1997,pach2004improving} showed in 1997 that $\kappa$ is bounded from above by this constant by considering graphs drawn on the plane as follows: take a small perturbation of the $\sqrt{n} \times \sqrt{n}$ grid and connect by a segments any two vertices at distance at most some suitable threshold $t$. Czabarka, Singgih, Sz\'ekely and Wang~\cite{CzabarkaSinggihSzekelyWang} later gave a clean verification of the same upper bound using drawings on the sphere where the vertices are distributed uniformly and nearby vertices are connected via spherical geodesics. Several authors have since asked whether this bound is asymptotically optimal ~\cite{CzabarkaSinggihSzekelyWang,toth2025generalizationscrossinglemma,schaefer2013survey}.

As the first of our two main results, we present a construction which performs asymptotically better in the midrange regime than the one from~\cite{PachToth1997}, thus proving the following.

\begin{theorem}\label{thm:main_counterexample}
The midrange crossing constant $\kappa$ satisfies 
\begin{equation}\label{eq:main_example}
\kappa<\frac{8}{9\pi^2}.
\end{equation}
\end{theorem}
This construction, which is described in the statement of Proposition~\ref{prop:nonuniform-counterexample}, was found by ChatGPT 5.6 Sol after we provided it with an earlier version of our manuscript.\footnote{The version of our manuscript that we fed into ChatGPT was uploaded to arXiv multiple months before the current version was made public and can still be found on this site.}$^,$\footnote{The numerical improvement over $8/(9\pi^2)$ is at least $10^{-27}$. We have made no effort in this manuscript to optimize this constant. We have, in fact, found more intricate constructions that significantly outperform it.}

For our second main result, we focus on a restricted setting in which the edges are represented by geodesic curves and the vertices are distributed ``uniformly'' over some domain. We show that among all such drawings, the midrange crossing constant is indeed $8/(9\pi^2)$. This may explain why it took nearly thirty years to improve the bound upper bound on $\kappa$ in the unrestricted setting.

We will define this notion of uniformity and state our results more precisely in the next subsection. For now, we state a continuous version of our theorems, which show the essence of our contribution.

\begin{theorem}\label{thm:intro-planar-main}
    Let $\Omega \subset {\mathbb R}^2$ be a compact convex set.  Let $w : \Omega \times \Omega \to [0,1]$ be arbitrary symmetric function. Let random points $x_1, x_2, x_3, x_4 \in \Omega$ be sampled independently and uniformly. Then, 
       \begin{equation*}
    {\mathbb E}\Bigl[{w(x_1,x_2)w(x_3,x_4)\II{\{[x_1x_2]\text{ crosses }[x_3x_4]\}}\Bigr]} \geq  \frac{8}{9\pi^2} \,  {\mathbb E}[w(x_1,x_2)]^3\,,
    \end{equation*} 
    where $[xy]$ denotes the segment with endpoints $x$ and $y$.
\end{theorem}

\begin{theorem}\label{thm:intro-spherical-main}
    Suppose $w : {\mathbb S}^2 \times {\mathbb S}^2 \to [0,1]$ is arbitrary symmetric measurable function. Let random points $x_1, x_2, x_3, x_4 \in {\mathbb S}^2$ be sampled independently and uniformly. Then, 
    \begin{equation*}
    {\mathbb E}\Bigl[{w(x_1,x_2)w(x_3,x_4)\II{\{[x_1x_2]\text{ crosses }[x_3x_4]\}}\Bigr]} \geq  \frac{(\sin t-t\cos t)^2}{8\pi^2}\,,
    \end{equation*}
where $t\in[0,\pi]$ satisfies $\cos t = 1-2 {\mathbb E}[w(x_1,x_2)]$. Equality is attained by $ w(x,y):=\II{\{d(x,y)\le t\}}.$ Here, $[xy]$ denotes the short geodesic arc connecting $x$ and $y$, and $d(x,y)$ the length of this arc.
\end{theorem}

Observe that in the spherical setting, as the edge density approaches $0$, the lower bound constant tends to $8/(9\pi^2)$. To the best of our knowledge, this is one of the first non-trivial classes of drawings where we are able to prove sharp asymptotic results as the density approaches zero.\footnote{The only other example that we know of comes from~\cite{twisted_midrange}, where the highly restrictive class of so-called ``twisted drawings'' was studied.}

At a technical level, the main contribution of this paper is the introduction and development of an analytic framework for studying drawings of dense graphs (i.e., graphs with at least $cn^2$ edges, for some arbitrary but fixed $c>0$) where all edges are represented by either segments on the plane or geodesic arcs on the sphere. This framework plays also a crucial role in describing and analyzing the construction used in the proof of Theorem~\ref{thm:main_counterexample}. 

In order to additionally showcase the effectiveness of our methods, we also determine the midrange crossing constant for the class of two-layer drawings, thus resolving an open question from~\cite{2-layer}; the precise statement and definitions will be given in the next section.

\subsection{Restricted drawing classes and main results}

Apart from Theorem~\ref{thm:main_counterexample}, our main results consist of new lower bounds for the midrange crossing constant for some restricted classes of drawings and graphs.
We shall work exclusively with spherical geodesic drawings and planar straight-line drawings, defined below.

\begin{definition}[Spherical geodesic drawings and straight-line drawings]\label{def:spherical-geodesic-drawing}
A \emph{spherical geodesic drawing} of a graph $G=(V,E)$ is a good drawing where the vertices are represented using an injective map
\[
    p:V\to {\mathbb S}^2
\]
such that no adjacent pair is mapped to antipodal points, and each edge $uv\in E$ is drawn as the shorter geodesic segment $[p(u)p(v)]$. 

Similarly, a \emph{straight-line drawing} of $G$ is a good drawing where the vertices are represented using an injective map $p:V\to\mathbb R^2$, and the edges are drawn as segments. Note that spherical geodesic drawings and straight-line drawings are completely determined by their defining map $p$, and thus we denote them simply by ${\mathcal D} = (G,p)$.
\end{definition}

\begin{definition}[Uniform spherical geodesic drawing sequence]\label{def:uniform-spherical-geodesic-sequence}
Let ${\mathcal D}_n=(G_n,p_n)$ be spherical geodesic drawings with $|V(G_n)|=n$.  The sequence $\{{\mathcal D}_n\}$ is \emph{uniform on the sphere} if
\[
        \nu_n:=\frac1n\sum_{v\in V(G_n)}\delta_{p_n(v)}\quad\to\quad\mu:=\frac{\sigma}{4\pi}
\]
weakly as probability measures on ${\mathbb S}^2$. Here $\sigma$ denotes spherical area measure.
\end{definition}

\begin{definition}[Uniform planar straight-line drawing sequence]\label{def:uniform-planar-drawing-sequence}
Let $\Omega\subset{\mathbb R}^2$ be a compact convex set with nonempty interior, and let
\[
        \lambda:=\frac{1}{\operatorname{area}(\Omega)}\,\mathcal L^2\big|_\Omega
\]
be the normalized area measure on $\Omega$ (i.e. the uniform probability measure on $\Omega$). A sequence of straight-line drawings ${\mathcal D}_n=(G_n,p_n)$ with $|V(G_n)|=n$ is \emph{uniform in $\Omega$} if
\[
        \frac1n\sum_{v\in V(G_n)}\delta_{p_n(v)}\quad\to\quad\lambda
\]
weakly.
\end{definition}

We now state our two main lower bounds for crossing numbers.

\begin{theorem}\label{thm:main-spherical}
Let ${\mathcal D}_n=(G_n,p_n)$ be a uniform sequence of spherical geodesic drawings, with $V(G_n)=n$ and asymptotic edge density $2|E(G_n)|/n^2 = e + o(1)$.  Let $t\in[0,\pi]$ be determined by $\cos t = 1-2e.$ Then,
\begin{equation}\label{eq:main-spherical-sharp-density}
        \operatorname{cr}({\mathcal D}_n)\ge\frac{n^4}{64\pi^2}\bigl(\sin t-t\cos t\bigr)^2-o(n^4)\,.
\end{equation}
\end{theorem}

\begin{theorem}\label{thm:main-planar}
Let ${\mathcal D}_n=(G_n,p_n)$ be a uniform sequence of straight-line drawings on a compact convex $\Omega \subset {\mathbb R}^2$ with , with $V(G_n)=n$ and $|E(G_n)|/n^2 = e + o(1)$.  Then,
\begin{equation}\label{eq:main-planar-sharp-density}
        \operatorname{cr}({\mathcal D}_n)\ge \frac{8n^4}{9\pi^2} e^3  - o(n^4)\,.
\end{equation}
\end{theorem}

In addition to the above, we will also consider the class of two-layer drawings, defined below.

\begin{definition}
    Given a bipartite graph $G$ with vertex classes $A$ and $B$, a \textit{two-layer drawing} of $G$ is a straight-line drawing where all vertices of $A$ are represented by points on a line $\ell_A$, and all vertices of $B$ are represented by points on another line $\ell_B$ which is parallel to $\ell_A$. 
\end{definition}

The class of two-layer drawings of bipartite graphs has been studied by several authors~\cite{2-layer_minimization,bipartite_crossing_number,efficient_1-sided,fpt_2-layer,stegosaurus,2-layer,wood_2-layer}. Importantly for us, Angelini, Da Lozzo, F{\"o}rster, and Schneck~\cite{2-layer} showed that a two-layer drawing of an $n$-vertex graph with $m\geq\frac{125}{48}$ edges must have at least \[\frac{4.608}{15.625}\frac{m^3}{n^2}\] crossings. In this paper, we exactly determine the midrange crossing constant for this class of drawings.
Let $\kappa_{\parallel}(n,m)$ denote the least number of crossings among two-layer drawings of bipartite graphs $n$ vertices and $m$ edges. We emphasize that no hypothesis on the asymptotic distribution of the vertices is made here.

\begin{theorem}\label{thm:2-layer}
Along every sequence of pairs $(n,m)$ with $n\ll m\ll n^2$,  \[\lim_{n\to\infty}\kappa_{\parallel}(n,m)\frac{n^2}{m^3}=\frac23.\]
\end{theorem}

The fact that the midrange crossing constant for the class of two-layer drawings is upper bounded by $2/3$ was already clear in~\cite{2-layer}; our contribution lies in the matching lower bound.

\subsection{Comparison to prior works} 

cdccdAs mentioned earlier, the best known lower bound for $\kappa$ follows from the work of B{\"u}ngener, Franz, Kaufmann and Pfister~\cite{bungener2026five}, who showed that an $n$-vertex graph with $m\geq7.39n$ edges must have crossing number at least \[\frac{648}{17689}\cdot\frac{m^3}{n^2}\approx\frac{1}{27.3}\cdot\frac{m^3}{n^2}\,.\] 

As far as we are aware, all previous works obtaining lower bounds for $\kappa$ have followed the same overarching strategy. This strategy consists of first deriving a universal lower bound for $\mathrm{cr}(G)$ which grows linearly in the number of edges $m$, and then \emph{boosting} this bound using the probabilistic method. Proving the linear lower bound often reduces to analyzing sparse graphs with $m = \Theta(n)$.
For example, the bound in~\eqref{eq:crossing_lemma} can be derived in this way by starting from the inequality $\mathrm{cr}(G)\geq m-(3n-6)$. In turn, this linear bound follows from the fact that a maximal planar subgraph of $G$ can have no more than $3n-6$ edges, and thus any additional edge must incur at least one crossing. For comparison, the stronger bound from~\cite{bungener2024improving} is obtained by first bounding the maximum size of a subgraph of $G$ which can be drawn without inducing certain forbidden configurations (e.g. an edge involved in at least $4$ crossings). 
If, as we suspect, $\kappa$ turns out to be close to $8/(9\pi^2)$, it seems very unlikely to us that this strategy will ever allow us to determine the precise value of this constant. 

Instead of working with graphs such that $m = \Theta(n)$, we will work almost exclusively with dense graphs with $m = \Theta(n^2)$; in other words, we approach the problem of determining the midrange crossing constant from the dense side, instead of the sparse one. More precisely, for $e\in(0,1]$, let us write $\tau(n,e)$ to denote the least crossing number of an $n$-vertex graph with $m= \lceil e\binom{n}{2}\rceil$ edges. It is known (and easy to prove) that, as $n\to \infty$,
\[\tau(n,e)\frac{n^2}{m^3}\sim\tau(n,e)\frac{8}{e^3n^4} \to \tau(e)\]
for some constant $\tau(e) > 0$. In turn, \Cref{thm:midrange_exists} tells us that \[\lim_{e\rightarrow0}\tau(e)=\kappa\,.\]
This simple identity presents us with a different approach to determining the exact value of $\kappa$.

Working with dense graphs has the important advantage of allowing us to operate in a continuous setting where, instead of studying finite graphs, our vertex set is represented by some probability measure $\mu$ on the sphere (or the plane), adjacencies are encoded by means of a measurable symmetric function $w$, and edges are represented by curves satisfying also some measurability condition.\footnote{Since we shall consider only drawings where edges are represented by geodesics, no additional measurability assumption is required on the curves representing the edges. However, the general would be more complicated.} For the reader who is familiar with the theory of dense graph limits, we clarify that this should simply be thought of as a drawing of a graphon on the sphere. Working with such continuous objects will be crucial in our proof, as it allows us to attack the problem from a more analytical perspective. In fact, our main arguments will all be carried out entirely in the continuous world, and only then will we translate our results back to the discrete setting. 

It is also of interest to determine the midrange crossing constant for other restricted classes of graphs and drawings. As discussed above, we fully resolve this problem not only for the class of uniform drawings on either the plane or the sphere, but also for the class of two-layer drawings of bipartite graphs. To the best of our knowledge, there is only one other non-trivial restricted class of drawings for which the midrange crossing constant has been determined exactly; namely, \textit{twisted graphs}~\cite{twisted_midrange}. 

Finally, we remark that the literature on crossing numbers is vast, and our discussion above has touched upon only one of the many interesting directions of research regarding this and other related parameters. We refer the reader to Schaefer's book~\cite{SchaeferSurvey} for further background and several interesting open problems about crossing numbers and their variants.

\subsection{Future directions}
Determining the midrange crossing constant is a central problem in geometric graph theory. In view of Theorem~\ref{thm:main_counterexample} (and its proof), we now have no reasonable conjecture regarding the value of $\kappa$. We believe it would be of interest to try and improve our upper bound for $\kappa$ as much as possible.

It seems reasonable to expect that our approach, together with some other ideas, could be applied to determine the midrange crossing constant for other restricted classes of drawings. A simple class for which the answer appears to not be known is that of \textit{convex drawings}; these are drawings where the vertices are represented by points in convex position and the edges are drawn as segments.

As we have mentioned, our methods are analytic and geometric in nature. Our proof of theorems~\ref{thm:main-spherical} and~\ref{thm:main-planar}, which from a technical point of view is the most intricate part of our work, proceeds through a sequence of inequalities which describe the optimal local geometry of a drawing. In particular, a certain version of Busemann's inequality and a rearrangement inequality identify the optimal local arrangement geodesic arcs passing through any point. Then, H\"older's turns these local inequalities into a global lower bound. In the uniform setting these ingredients fit together sharply, and the equality case corresponds exactly to the threshold construction. We hope that these analytic techniques can be made robust enough to treat more general situations.

Lastly, in this paper we work exclusively with spherical geodesic drawings and planar straight-line drawings. Extending our tools to more general classes of drawings where edges are be represented by other kinds of curves may be possible, although perhaps quite challenging.

\subsection*{Structure of the paper}
Sections~\ref{sec:main},~\ref{sec:busemann} and~\ref{sec:reduction} are devoted to the proofs of theorems~\ref{thm:main-spherical} and~\ref{thm:main-planar}: in Section~\ref{sec:main}, we study drawings in a continuous setting and prove theorems~\ref{thm:intro-planar-main} and~\ref{thm:intro-spherical-main} conditionally on a generalized version of an inequality of Busemann; Section~\ref{sec:busemann} contains the proof of this inequality; in Section~\ref{sec:reduction} we transfer our results to the discrete setting. In Section~\ref{sec:nonuniform-counterexample} we present a constructive proof of Theorem~\ref{thm:main_counterexample}. Finally, the proof of Theorem~\ref{thm:2-layer} appears in Section~\ref{sec:two-parallel-lines}.

\subsection*{Acknowledgments}
O.S.P. is very grateful to Larry Guth, Arnaud de Mesmay, Hugo Parlier and, particularly, Alfredo Hubard, for several insightful conversations. S.L. and O.S.P. were partially supported by the NSF–Simons Research Collaboration Grant (Award No.~2031883).

\subsection{Statement on AI use}

The first version of this manuscript, available on arXiv under the identifier 2605.16700v1, contained the proofs of \cref{thm:main-spherical,thm:main-planar}, but not those of \cref{thm:main_counterexample,thm:2-layer}. \Cref{sec:main,sec:busemann,sec:reduction} remain essentially unchanged from that version. The proofs in that version were developed entirely by the authors, without assistance from any AI model.

In our attempt to relax the uniformity assumption in \cref{thm:main-planar}, we provided the earlier version of the manuscript to ChatGPT 5.6 Sol and asked it to prove or disprove the corresponding statement for smooth densities. After a few iterations, ChatGPT found a counterexample using the same function $w$ and a periodic density of the form $\rho = 1 + \varepsilon \cos(kx)$, with small $\varepsilon$ and large $k$. We subsequently refined this counterexample for presentation, obtaining \Cref{thm:main_counterexample}.

We later added the proof of \Cref{thm:2-layer} to further illustrate the usefulness of our analytic approach. This argument was developed entirely by the authors, without AI assistance.

ChatGPT was also used to produce figures and to edit and polish both versions of the manuscript .

\section{Uniform continuous drawings}\label{sec:main}

The analytic part of the proof of theorems~\ref{thm:main-spherical} and~\ref{thm:main-planar} consists of a continuous inequality for weighted geodesic edges on the sphere.  The finite drawing ${\mathcal D}_n$ will later be converted into such a weighted object in \Cref{sec:reduction}.

\subsection{Main Theorem}\label{subsec:main-setup}
Let ${\mathbb S}^2\subset{\mathbb R}^3$ be the unit sphere, let $\sigma$ denote the area measure, and
\[
        \mu:=\frac{\sigma}{4\pi}
\]
be the normalized spherical area measure.  For $x,y\in{\mathbb S}^2$, write $d(x,y)$ for the spherical distance.  When $x$ and $y$ are not antipodal, write $[xy]$ for the shorter geodesic segment joining $x$ to $y$.  A continuous geodesic spherical drawing is encoded by a measurable symmetric function $w:{\mathbb S}^2\times{\mathbb S}^2\to[0,1]$. Its ordered edge density is
\begin{equation}\label{eq:continuous-density}
        e(w):=
        \int_{{\mathbb S}^2}\int_{{\mathbb S}^2}w(x,y)\,\mu(\dd x)\mu(\dd y)\,,
\end{equation}
and its crossing functional is
\begin{equation}\label{eq:continuous-crossing-functional}
        \operatorname{Cr}(w):=\int_{({\mathbb S}^2)^4}w(x_1,x_2)w(x_3,x_4)\II{\{[x_1x_2]\text{ crosses }[x_3x_4]\}}\prod_{i=1}^4\mu(\dd x_i)\,.
\end{equation}

The crossing condition in \eqref{eq:continuous-crossing-functional} is understood up to the null set of degenerate configurations. The purpose of this section is to prove the following theorem.

\begin{theorem}\label{thm:continuous-spherical-crossing-main}
Let $w:{\mathbb S}^2\times{\mathbb S}^2\to[0,1]$ be measurable and symmetric. Then,
\begin{equation}\label{eq:continuous-spherical-bound-main}
        \operatorname{Cr}(w)\ge\frac{(\sin t-t\cos t)^2}{8\pi^2}\,,
\end{equation}
where $t\in[0,\pi]$ satisfies $\cos t = 1-2e(w)$. Equality is attained by $ w(x,y):=\II{\{d(x,y)\le t\}}.$ 
\end{theorem}

Observe that the spherical cap of radius $t$ has normalized area precisely $(1-\cos t)/2$, which corresponds to average degree or ordered edge density  above.

\medskip

We now collect the notation and technical lemmas required for the proof. We parametrize crossing geodesic edges by their intersection point and tangent directions.  Fix once and for all a measurable orthonormal frame $(e_1(x),e_2(x))$ for $T_x{\mathbb S}^2$.  For $\theta\in[0,2\pi)$, set $v_{x,\theta}:=\cos\theta\,e_1(x)+\sin\theta\,e_2(x).$ Set also
\begin{equation}\label{eq:triangle-domain-main}
        \Delta:=\{(s,r):s>0,\ r>0,\ s+r<\pi\}.
\end{equation}
For $(s,r)\in\Delta,$ the endpoint pair $ \left(\exp_x(sv_{x,\theta}),\exp_x(-rv_{x,\theta})\right)$ is the oriented geodesic segment passing through $x$ in direction $\theta$, with $x$ lying between the two endpoints. After fixing $w$, define the crossing flux and incidence flux as follows.
\begin{definition}\label{def:normalized-spherical-fluxes}
For $x\in{\mathbb S}^2$ and $\theta\in[0,2\pi)$, we define
\begin{align}
        g(x,\theta)&:=\frac1{(4\pi)^2}\int_{\Delta}\sin(s+r)\,w\left(\exp_x(sv_{x,\theta}),\exp_x(-rv_{x,\theta})\right)\,\dd r\,\dd s \,, \label{eq:normalized-g-def}\\
        a(x,\theta)&:=\frac1{(4\pi)^2}\int_{\Delta}\frac{\sin(s+r)}{s+r}\,w\left(\exp_x(sv_{x,\theta}),\exp_x(-rv_{x,\theta})\right)\,\dd r\,\dd s\,.\label{eq:normalized-a-def}
\end{align}
\end{definition}

\begin{figure}[t]
\centering
\begin{tikzpicture}[scale=0.7, every node/.style={font=\small}]
        \def\striplen{3.7}
        \def\stripwidth{0.48}

        \begin{scope}[rotate=18]
                \fill[blue!12] (-\striplen,-\stripwidth) rectangle (\striplen,\stripwidth);
        \end{scope}
        \begin{scope}[rotate=-32]
                \fill[red!10] (-\striplen,-\stripwidth) rectangle (\striplen,\stripwidth);
        \end{scope}

        \begin{scope}
                \clip[rotate=18] (-\striplen,-\stripwidth) rectangle (\striplen,\stripwidth);
                \fill[gray!25, rotate=-32] (-\striplen,-\stripwidth) rectangle (\striplen,\stripwidth);
        \end{scope}

        \begin{scope}[rotate=18]
                \foreach \y in {-0.36,-0.24,-0.12,0,0.12,0.24,0.36}
                        \draw[blue!55!black, line width=0.45pt] (-\striplen,\y)--(\striplen,\y);
                \draw[blue!80!black, line width=0.85pt] (-\striplen,-\stripwidth)--(\striplen,-\stripwidth);
                \draw[blue!80!black, line width=0.85pt] (-\striplen,\stripwidth)--(\striplen,\stripwidth);
        \end{scope}
        \begin{scope}[rotate=-32]
                \foreach \y in {-0.36,-0.24,-0.12,0,0.12,0.24,0.36}
                        \draw[red!55!black, line width=0.45pt] (-\striplen,\y)--(\striplen,\y);
                \draw[red!80!black, line width=0.85pt] (-\striplen,-\stripwidth)--(\striplen,-\stripwidth);
                \draw[red!80!black, line width=0.85pt] (-\striplen,\stripwidth)--(\striplen,\stripwidth);
        \end{scope}

        \draw[->, line width=0.85pt] (0,0)--(18:1.55) node[pos=0.94, above, inner sep=1pt] {};
        \draw[->, line width=0.85pt] (0,0)--(-32:1.55) node[pos=0.94, below, inner sep=1pt] {};
        \draw[->] (18:0.68) arc[start angle=18, end angle=-32, radius=0.68];

        \node[blue!70!black, anchor=west, fill=white, inner sep=1pt] at (2.55,0.95) {edges in $\theta$ direction};
        \node[red!70!black, anchor=west, fill=white, inner sep=1pt] at (1.95,-1.58) {edges in $\varphi$ direction};
\end{tikzpicture}
\caption{Two streams of edges at a crossing point $x$.  The two strips overlap in a parallelogram, and the area element of this parallelogram  gives the factor $|\sin(\theta_1 -\theta_2)|$.}
\label{fig:parallelogram}
\end{figure}

The Jacobian change of variables computation, deferred to \Cref{subsec:main-lemmas}, gives the following representation for the crossing functional
\begin{equation}\label{eq:normalized-flux-crossing-representation}
        \operatorname{Cr}(w)=\int_{{\mathbb S}^2}\int_0^{2\pi}\int_0^{2\pi}g(x,\theta)g(x,\phi)\abs{\sin(\theta-\phi)}\,\dd\theta\,\dd\phi\,\sigma(\dd x)\,.
\end{equation}
Pictorially, \Cref{fig:parallelogram} explains why $|\sin(\theta - \phi)|$ shows up in \eqref{eq:normalized-flux-crossing-representation}. Similarly, we get  
\begin{equation}\label{eq:normalized-incidence-identity}
       e(w)= \int_{{\mathbb S}^2}\int_0^{2\pi}a(x,\theta)\,\dd\theta\,\sigma(\dd x)\,.
\end{equation}
We show that one can lower bound $g$ by a function of $a$ pointwise as follows.
\begin{lemma}\label{lem:normalized-bathtub-main}
For almost every $(x,\theta) \in {\mathbb S}^2 \times [0, 2\pi)$, we have 
\begin{equation}\label{eq:normalized-pointwise-bathtub}
    g(x,\theta) \geq \varphi (a (x,\theta))\,,
\end{equation}
where $\varphi : [0,1/(8\pi^2)] \to {\mathbb R}$ is defined by 
\begin{equation}\label{eq:normalized-bathtub-profile}
        \varphi(\alpha):=\frac{\sin \tau(\alpha)-\tau(\alpha)\cos \tau(\alpha)}{(4\pi)^2}\,,
\end{equation}
where $\tau(\alpha) \in [0,\pi]$ is determined by $16\pi^2 \alpha =1-\cos  \tau(\alpha).$
\end{lemma}

We will also use that $\alpha \mapsto  \varphi(\alpha)^{2/3}$ is convex on its domain. The proof of \Cref{lem:normalized-bathtub-main} and the convexity calculation are deferred to \Cref{subsec:main-lemmas}. Finally, the most crucial part of the argument will be a version of Busemann's inequality that we will prove in \Cref{sec:busemann}.

\begin{lemma}[Busemann inequality]\label{lem:busemann-circle-main}
For every measurable $f:[0,2\pi)\to[0,\infty)$,
\begin{equation}\label{eq:busemann-circle-main}
        \int_0^{2\pi}\int_0^{2\pi}f(\theta)f(\phi)\abs{\sin(\theta-\phi)}\,\dd\theta\,\dd\phi\ge\frac1{\pi^2}\left(\int_0^{2\pi}f(\theta)^{2/3}\,\dd\theta\right)^3.
\end{equation}
\end{lemma}

We now combine all the ingredients above to prove \Cref{thm:continuous-spherical-crossing-main}.

\begin{proof}[Proof of \Cref{thm:continuous-spherical-crossing-main}]
The crossing representation \eqref{eq:normalized-flux-crossing-representation} and Busemann's inequality \eqref{eq:busemann-circle-main} give
\begin{align*}
    \operatorname{Cr}(w)&=\int_{{\mathbb S}^2}\int_0^{2\pi}\int_0^{2\pi}
        g(x,\theta)g(x,\phi)\abs{\sin(\theta-\phi)}\,\dd\theta\,\dd\phi\,\sigma(\dd x) \\
        &\geq  \frac1{\pi^2}\int_{{\mathbb S}^2} \left(\int_0^{2\pi}g(x,\theta)^{2/3}\,\dd\theta\right)^3 \sigma(\dd x)\,.
\end{align*}
H\"older's inequality on $({\mathbb S}^2,\sigma)$ with $\sigma({\mathbb S}^2)=4\pi$ yields
\begin{align*}
    \operatorname{Cr}(w)&\geq  \frac1{\pi^2}\int_{{\mathbb S}^2} \left(\int_0^{2\pi}g(x,\theta)^{2/3}\,\dd\theta\right)^3 \sigma(\dd x) \geq \frac1{16\pi^4}\left(\int_{{\mathbb S}^2} \int_0^{2\pi}g(x,\theta)^{2/3}\,\dd\theta \,\sigma(\dd x) \right)^3.
\end{align*}

Let $h(\alpha) = \varphi(\alpha)^{2/3}$. Then, the pointwise estimate \Cref{lem:normalized-bathtub-main} gives 
\[
        g(x,\theta)^{2/3}\ge h(a(x,\theta))\,.
\]
By the convexity of $h$, Jensen's inequality on the space
${\mathbb S}^2\times[0,2\pi)$ implies
\begin{equation}\label{eq:main-proof-jensen-step}   \int_{{\mathbb S}^2}\int_0^{2\pi}g(x,\theta)^{2/3}\,\dd\theta\,\sigma(\dd x)
        \ge8\pi^2\,h\left(\frac1{8\pi^2}\int_{{\mathbb S}^2}\int_0^{2\pi}a(x,\theta)\,\dd\theta\,\sigma(\dd x)\right)  =8\pi^2\,h\left(\frac{e(w)}{8\pi^2}\right)\,,
\end{equation}
where the last step uses the identity \eqref{eq:normalized-incidence-identity}.
Let $t\in[0,\pi]$ be determined by $\cos t=1-2e(w)$.  The definition of $\varphi$ gives
\[
        h\left(\frac{e(w)}{8\pi^2}\right)=\left(\frac{\sin t-t\cos t}{(4\pi)^2}\right)^{2/3}\,.
\]
Substituting this into the preceding bound and using \eqref{eq:main-proof-jensen-step}, we obtain
\[
        \operatorname{Cr}(w)\ge\frac1{16\pi^4}(8\pi^2)^3\left(\frac{\sin t-t\cos t}{(4\pi)^2}\right)^2=\frac{(\sin t-t\cos t)^2}{8\pi^2}\,.
\]
This proves the lower bound. It remains to check sharpness. 

For density $w(x,y)=\II{\{d(x,y)\le t\}}$, we have $e(w)=(1-\cos t)/2$.  Moreover, we have that
\[
        g(x,\theta)=\frac1{(4\pi)^2}\int_0^t u\sin u\,\dd u=\frac{\sin t-t\cos t}{(4\pi)^2}
\]
independently of $(x,\theta)$.  Since $\sigma({\mathbb S}^2)=4\pi$ and 
\[
        \int_0^{2\pi}\int_0^{2\pi}\abs{\sin(\theta-\phi)}\,\dd\theta\,\dd\phi=8\pi\,,
\]
the flux formula \eqref{eq:normalized-flux-crossing-representation} gives
\[
        \operatorname{Cr}(w)=32\pi^2\left(\frac{\sin t-t\cos t}{(4\pi)^2}\right)^2=\frac{(\sin t-t\cos t)^2}{8\pi^2}\,.
\]
Hence equality is attained by $w$.
\end{proof}

\subsection{Proofs of Lemmas}\label{subsec:main-lemmas}
We now give the deferred details for the geometric identities and the one dimensional rearrangement estimate used in the proof of \Cref{thm:continuous-spherical-crossing-main}.  The first two lemmas are change of variables computations for parametrizing a geodesic segment by an interior point, a tangent direction, and the distances to the two endpoint.  The third lemma reduces the pointwise estimate $g\ge \varphi(a)$ on each geodesic fiber to the bathtub principle.  

\begin{lemma}[Crossing identity]\label{lem:spherical-crossing-jacobian-main}
Let $w:{\mathbb S}^2\times{\mathbb S}^2\to[0,1]$ be measurable and symmetric, and let $g$ be defined by \eqref{eq:normalized-g-def}.  Then, we have that
\begin{equation*}
        \operatorname{Cr}(w)=\int_{{\mathbb S}^2}\int_0^{2\pi}\int_0^{2\pi}g(x,\theta)g(x,\phi)\abs{\sin(\theta-\phi)}\,\dd\theta\,\dd\phi\,\sigma(\dd x)\,.
\end{equation*}
\end{lemma}

\begin{proof}
For $i=1,2$ and $(s_i,r_i)\in\Delta $, set $ y_i^+=\exp_x(s_i v_{x,\theta_i})$ and $y_i^-=\exp_x(-r_i v_{x,\theta_i})$. The map
\[
        (x,\theta_1,\theta_2,s_1,r_1,s_2,r_2)\longmapsto (y_1^+,y_1^-,y_2^+,y_2^-)
\]
parametrizes, up to null sets, the ordered endpoint quadruples whose two geodesic segments $[y_1^-y_1^+]$ and $[y_2^-y_2^+]$ cross transversely at $x$. Here $\theta_1$ and $\theta_2$ denote the directions through $x$ of the geodesic segments $[y_1^-y_1^+]$ and $[y_2^-y_2^+]$, respectively, while $s_i$ and $r_i$ represent the signed distances from $x$ to the corresponding endpoints. It now remains to compute its Jacobian. Work locally in a smooth oriented orthonormal frame and write
\[
        v_i:=v_{x,\theta_i}\,,\qquad n_i:=\partial_{\theta_i}v_{x,\theta_i}\,.
\]
For one segment, if $\xi\in T_x{\mathbb S}^2$ is the base point change and $\eta_i$ is the angular change, then the changed endpoint coordinates are
\[
        \cos s_i\,\langle\xi,n_i\rangle+\eta_i\sin s_i\,,\qquad\cos r_i\,\langle\xi,n_i\rangle-\eta_i\sin r_i\,.
\]
Hence the change determinant for the $i$-th segment is
\[
        \abs{\det\begin{pmatrix}\cos s_i & \sin s_i\\\cos r_i & -\sin r_i\end{pmatrix}}=\sin(s_i+r_i)\,,
\]
while $s_i$ and $r_i$ give the longitudinal directions.  The remaining base point change of coordinates contributes $\abs{\det(n_1,n_2)}=\abs{\sin(\theta_1-\theta_2)}.$ Therefore,
\begin{equation}\label{eq:spherical-crossing-jacobian-main}
        J=\sin(s_1+r_1)\sin(s_2+r_2)\abs{\sin(\theta_1-\theta_2)}\,.
\end{equation}

Applying this to \eqref{eq:continuous-crossing-functional} and using $\mu=\sigma/(4\pi)$ gives
\begin{align*}
        \operatorname{Cr}(w)
        &=\frac{1}{(4\pi)^{4}}\int_{{\mathbb S}^2}\int_0^{2\pi}\int_0^{2\pi}\int_{\Delta}\int_{\Delta}w(y_1^+,y_1^-)w(y_2^+,y_2^-) J\,\dd r_1\,\dd s_1\,\dd r_2\,\dd s_2\,\dd\theta_1\,\dd\theta_2\,\sigma(\dd x) \\
        &=\int_{{\mathbb S}^2}\int_0^{2\pi}\int_0^{2\pi}g(x,\theta_1)g(x,\theta_2)\abs{\sin(\theta_1-\theta_2)}\,\dd\theta_1\,\dd\theta_2\,\sigma(\dd x)\,,
\end{align*}
which proves the claim.
\end{proof}

\begin{lemma}[Edge identity]\label{lem:incidence-identity-main}
Let $w:{\mathbb S}^2\times{\mathbb S}^2\to[0,1]$ be measurable and symmetric, and let $a$ be defined by \eqref{eq:normalized-a-def}.  Then, we have that
\[
        e(w)=\int_{{\mathbb S}^2}\int_0^{2\pi}a(x,\theta)\,\dd\theta\,\sigma(\dd x)\,.
\]
\end{lemma}

\begin{proof}
From the proof of \Cref{lem:spherical-crossing-jacobian-main}, the normal Jacobian of the map
\[
        (x,\theta,s,r)\longmapsto\left(\exp_x(sv_{x,\theta}),\exp_x(-rv_{x,\theta})\right)
\]
is precisely $\sin(s+r)$.  For fixed endpoints $(y,z)$, the fiber consists of the points $x$ on the segment $[yz]$, and its length is $d(y,z)=s+r$.  This gives
\begin{equation*}\int_{{\mathbb S}^2}\int_0^{2\pi}a(x,\theta)\,\dd\theta\,\sigma(\dd x)=
        \frac1{(4\pi)^2}\int_{{\mathbb S}^2}\int_{{\mathbb S}^2}w(y,z)\,\sigma(\dd y)\sigma(\dd z)=e(w)\,.\qedhere 
\end{equation*}
\end{proof}
\begin{lemma}\label{lem:spherical-pointwise-inequality}
Let $w:{\mathbb S}^2\times{\mathbb S}^2\to[0,1]$ be measurable and symmetric, and let $g$ and $a$ be defined by equations \eqref{eq:normalized-g-def}-\eqref{eq:normalized-a-def} and $\varphi$ be as  \eqref{eq:normalized-bathtub-profile}. Then, for almost every $(x,\theta) \in {\mathbb S}^2 \times [0,2\pi)$, we have
\begin{equation*}
    g(x,\theta)\geq \varphi(a(x,\theta))\,.
\end{equation*}
\end{lemma}

\begin{proof}
Fix $(x,\theta)$. For $(s,r) \in \Delta$  write $q(s,r): =w\left(\exp_x(sv_{x,\theta}),\exp_x(-rv_{x,\theta})\right)$. Let us do a change variables to $u=s+r$ and $z=s$.  For $u \in (0, \pi)$, put
\[
        b(u):=\int_0^u q(z,u-z)\,\dd z\,.
\]
Note that $b(u) \in [0,u]$. Now, equations \eqref{eq:normalized-g-def} and \eqref{eq:normalized-a-def} translate to 
\[
       a(x,\theta)= \frac{1}{ (4\pi)^2}\int_0^\pi \frac{\sin u}{u}b(u)\,\dd u\,,
        \qquad
       g(x,\theta)=\frac{1}{ (4\pi)^2}\int_0^\pi \sin u\,b(u)\,\dd u\,.
\]
Equivalently, let $\dd\nu(u)=(\sin u/u)b(u)\,\dd u$.  Then $0\le\dd\nu\le \sin u\,\dd u$, its mass is $(4\pi)^2a(x,\theta)$, and
\[
        g(x,\theta)=\frac{1}{(4\pi)^2}\int_0^\pi u\,\dd\nu(u)\,.
\]
Among all measures dominated by $\sin u\,\dd u$ with fixed mass, the integral of the increasing cost $u$ is minimized by filling the interval from the left.  Thus if $\tau$ is determined by
\[
        (4\pi)^2a(x,\theta)=\int_0^\tau \sin u\,\dd u=1-\cos \tau \,,
\]
then, we have that 
\[
        (4\pi)^2g(x,\theta)\ge\int_0^\tau  u\sin u\,\dd u=\sin \tau -\tau \cos \tau \,.
\]
This is precisely $g(x,\theta)\ge\varphi(a(x,\theta))$. Moreover, equality occurs when $\dd\nu=\sin u\,\dd u$ on $[0,\tau ]$ and $\dd\nu=0$ on $(\tau ,\pi]$, which is equivalent to $q(s,r)=\II{\{s+r\le \tau \}}$ for almost every $(s,r)$.
\end{proof}

\begin{lemma}\label{lem:normalized-convexity-main}
Let $\varphi$ be defined as in \eqref{eq:normalized-bathtub-profile}.
 Then, the function $\alpha \mapsto \varphi(\alpha)^{2/3}$ is convex on $[0,1/(8\pi^2)]$.
\end{lemma}

\begin{proof}
First, observe that by chain rule, we have 
\begin{equation*}
    \frac{\dd \varphi}{\dd \tau}(\tau) = \frac{\tau \sin \tau}{(4\pi)^2} \quad \text{ and } \quad \frac{\dd \alpha}{\dd \tau}(\tau) = \frac{\sin \tau}{(4\pi)^2} \; \;\implies \;\; \frac{\dd \varphi }{\dd \alpha}(\alpha) = \tau(\alpha)\,.
\end{equation*}
Let $h(\alpha) = \varphi(\alpha)^{2/3}$. Then, by chain rule again, we get
\[
       h''(\alpha)= \frac{2}{9}\,\varphi(\alpha)^{-4/3}\sqb{3\varphi(\alpha)\varphi ''(\alpha)  - \varphi'(\alpha)^2}\,.
\]
Thus, written in terms of $\tau \in (0,\pi)$, the convexity is equivalent to
\[
        3 (1 - \tau \cot \tau ) - \tau^2 \geq 0\,.
\]
Using Euler's partial fraction expansion for $\tau \in (0, \pi)$,
\[
        \cot \tau=\frac1\tau+2\tau\sum_{k=1}^{\infty}\frac1{\tau^2-k^2\pi^2} \implies 1-\tau\cot \tau=2\tau^2\sum_{k=1}^{\infty}\frac1{k^2\pi^2-\tau^2}\ge2\tau^2\sum_{k=1}^{\infty}\frac1{k^2\pi^2}=\frac{\tau^2}{3}\,.
\]
Therefore $h''(\alpha)\ge0$ in the interior of its domain.  Continuity at endpoints gives full convexity.
\end{proof}

\subsection{Planar Theorem}\label{subsec:main-planar}
The same mechanism also gives the Euclidean analogue of \Cref{thm:continuous-spherical-crossing-main}.  Let $\Omega\subset{\mathbb R}^2$ be compact and convex, let $A=:\operatorname{area}(\Omega)$, and write $\lambda:=A^{-1}\mathcal L^2|_\Omega$ for the normalized Lebesgue measure, or uniform measure.  For a measurable symmetric $w:\Omega\times\Omega\to[0,1]$, set
\[
        e_\Omega(w):=\int_{\Omega}\int_{\Omega}w(y,z)\,\lambda(\dd y)\lambda(\dd z)\,,
\]
and let $\operatorname{Cr}_\Omega(w)$ be the straight-line crossing functional on $\Omega$, defined analogously to \eqref{eq:continuous-crossing-functional}.

\begin{theorem}\label{thm:continuous-planar-crossing-main}
For every compact convex $\Omega\subset{\mathbb R}^2$ with nonempty interior and every measurable symmetric $w:\Omega\times\Omega\to[0,1]$, we have
\begin{equation}\label{eq:continuous-planar-crossing-main}
        \operatorname{Cr}_\Omega(w)\ge\frac{8}{9\pi^2}\,e_\Omega(w)^3\,.
\end{equation}
\end{theorem}

\begin{proof}
Let us extend $w$ to ${\mathbb R}^2 \times {\mathbb R}^2$ by $0$. For $\theta\in[0,2\pi)$, write $v_\theta=(\cos\theta,\sin\theta)$ and define
\begin{align*}
        g_\Omega(x,\theta)&:=\frac{1}{A^2}\int_{0}^\infty\int_{0}^\infty(s+r)w(x+sv_\theta,x-rv_\theta)\,\dd r\dd s\,,\\a_\Omega(x,\theta)
        &:=\frac{1}{A^2}\int_{0}^\infty\int_{0}^\infty w(x+sv_\theta,x-rv_\theta)\,\dd r\dd s\,.
\end{align*}
By the analogous computation to \Cref{lem:incidence-identity-main} and \Cref{lem:spherical-crossing-jacobian-main}, the Euclidean crossing  Jacobian is given by
$(s+r)(s'+r')\abs{\sin(\theta-\phi)}$, so we can write 
\[
\begin{aligned}
        \operatorname{Cr}_\Omega(w)
        &=\int_\Omega\int_0^{2\pi}\int_0^{2\pi}g_\Omega(x,\theta)g_\Omega(x,\phi)\abs{\sin(\theta-\phi)}\,\dd\theta \, \dd\phi \, \dd x\,,\\
        e_\Omega(w)
        &=\int_\Omega\int_0^{2\pi}a_\Omega(x,\theta)\,\dd\theta \, \dd x\,.
\end{aligned}
\]

As in the spherical proof, Busemann's inequality and H\"older give
\begin{equation}\label{eq:planar-holder-step-main}
        \operatorname{Cr}_\Omega(w)\ge\frac1{\pi^2A^2}\left(\int_\Omega\int_0^{2\pi}g_\Omega(x,\theta)^{2/3}\,\dd\theta \, \dd x\right)^3\,.
\end{equation}
For fixed $(x,\theta)$, analogously to \Cref{lem:spherical-pointwise-inequality} we can prove that 
\begin{equation*}
    g_\Omega(x,\theta) \geq \frac{2\sqrt{2} }{3} A  \, a_\Omega(x,\theta)^{3/2} \implies \int_\Omega\int_0^{2\pi}g_\Omega(x,\theta)^{2/3}\,\dd\theta\, \dd x \ge\left(\frac{2\sqrt2}{3} A\right)^{2/3}e_\Omega(w)\,.
\end{equation*}
Substitution into \eqref{eq:planar-holder-step-main} yields
\[
        \operatorname{Cr}_\Omega(w)\ge\frac1{\pi^2A^2}\left(\frac{2\sqrt2}{3} A\right)^2e_\Omega(w)^3= \frac8{9\pi^2} e_\Omega(w)^3\,. \qedhere \]
\end{proof}
One can think of \Cref{thm:continuous-planar-crossing-main} as a local version of \Cref{thm:continuous-spherical-crossing-main}. Indeed, after transferring a planar drawing to a sufficiently small local chart on the sphere, the planar lemmas used in the proof of \Cref{thm:continuous-planar-crossing-main} become immediate corollaries of the corresponding lemmas in the spherical setting. Moreover, if the plane can be tiled by copies of $\Omega$, then \Cref{thm:continuous-planar-crossing-main} can be deduced directly from \Cref{thm:continuous-spherical-crossing-main}: one covers the sphere with rescaled copies of the planar drawing, with the scales tending to zero, so that the resulting error tends to zero as well. 

\section{Busemann's inequality}\label{sec:busemann}

In this section, we prove Busemann’s inequality~\ref{lem:busemann-circle-main}. If one interprets $f^{1/3}$ as the radial function of a star-shaped body in the plane, then the inequality has a geometric meaning. It states that the expected area of the triangle formed by the center and two uniformly randomly sampled points from the body is minimized when the body is an ellipsoid. This statement for convex bodies is originally proved in \cite{Busemann53}. We give an alternative proof tailored to our purposes.  We start by proving a linear version of the inequality, and then we lift it up.

\begin{lemma}\label{lem:busemann-line}
For every measurable function $q:{\mathbb R}\to[0,\infty)$,
\begin{equation}\label{eq:busemann-line}
        \int_{{\mathbb R}}\int_{{\mathbb R}}q(u)q(v)\abs{u-v}\,\dd u\dd v
        \ge
        \frac{2}{\pi^2}
        \left(\int_{{\mathbb R}}q(u)^{2/3}\,\dd u\right)^3\,.
\end{equation}
\end{lemma}

\begin{proof}
By truncating $q$ to $q_n(u):=(q(u) \wedge n)\II{\{\abs{u}\le n\}}$ and using monotone convergence on both sides, it is enough to prove the estimate for bounded $q$ supported on a finite interval.  On such an interval we may approximate in $L^1$ by functions which are positive and continuous on their support; since the kernel $\abs{u-v}$ is bounded on a fixed compact set and since $|a^{2/3}-b^{2/3}|\le |{a-b}|^{2/3}$, the estimate passes to the limit.  Thus, without loss of generality, we assume that $q$ is positive and continuous on its support $[a,b]$, and that it integrates to $1$. Define
\[
        Q(x):=\int_a^x q(u)\,\dd u\,.
\]
Now we note that equality $|u - v| = \int_{{\mathbb R}}\II{\{u<x<v\}}+\II{\{v<x<u\}}\,\dd x$ lets us rewrite the left hand side as
\begin{equation}\label{eq:busemann-line-layercake}
\begin{aligned}
        \int_{{\mathbb R}}\int_{{\mathbb R}}q(u)q(v)\abs{u-v}\,\dd u\dd v
        &=2\int_a^b Q(x)(1-Q(x))\,\dd x\,.
\end{aligned}
\end{equation}
Define $x(t)$ by $Q(x(t))=t$ for $t\in(0,1)$.  Then $x'(t)=q(x(t))^{-1}$, so the equality \eqref{eq:busemann-line-layercake} becomes
\begin{equation}\label{eq:busemann-line-quantile-energy}
        \int_{{\mathbb R}}\int_{{\mathbb R}}q(u)q(v)\abs{u-v}\,\dd u\dd v=2 \int_0^1 t(1-t) q(x(t))^{-1}\,\dd t = 2\int_0^1 t(1-t) \, y(t)^3\,\dd t\,,
\end{equation}
where $y(t):=q(x(t))^{-1/3}$. Similarly, by the change of variables $u = x(t)$, we get 
\begin{equation}\label{eq:busemann-line-quantile-mass}
        \int_{{\mathbb R}}q(u)^{2/3}\,\dd u=\int_0^1 q(x(t))^{-1/3}\,\dd t = \int_0^1 y(t)\,\dd t\,.
\end{equation}
Now, H\"older's inequality gives
\[
        \Biggl({\int_0^1 y(t)\,\dd t\Biggr)}^3\le\Biggl(\int_0^1 t(1-t) \, y(t)^3\,\dd t\Biggr)\Biggl(\int_0^1 \frac{\dd t}{\sqrt{t(1-t)}}\Biggr)^{2} = \pi^2 \int_0^1 t(1-t) \, y(t)^3\,\dd t\,.
\]
Hence, we are done.
\end{proof}

\begin{proof}[Proof of \Cref{lem:busemann-circle-main}]
Let us define a bilinear form $B$ by
\[
        B(f,h):=\int_0^{2\pi}\int_0^{2\pi}f(\theta)h(\phi)\abs{\sin(\theta-\phi)}\,\dd\theta\dd\phi \,.
\]
First replace $f$ by the average $\bar f(\theta) := [f(\theta)+f(\theta+\pi)]/2,$ where angles are taken modulo $2\pi$.  The kernel $\abs{\sin(\theta-\phi)}$ is unchanged when either variable is shifted by $\pi$, so bilinearity gives $B(\bar f,\bar f)=B(f,f)$.  Since $x\mapsto x^{2/3}$ is concave, we have
\[
        \int_0^{2\pi} \bar f(\theta)^{2/3}\,\dd\theta\ge\int_0^{2\pi} f(\theta)^{2/3}\,\dd\theta\,.
\]
We can therefore assume without loss of generality that $f$ is $\pi$-periodic. Up to a null set of directions, write $u=\tan\theta$ with $\theta\in(-\pi/2,\pi/2)$, and define
\[
        q(u):=\frac{f(\tan^{-1} u)}{\,\,(1+u^2)^{3/2}}\,.
\]
If $u = \tan \theta$ and $v = \tan \phi$, we note that 
\[
        \abs{\sin(\theta-\phi)}=\frac{\abs{u-v}}{\sqrt{(1+u^2)(1+v^2)}}\quad \text{ and } \quad \frac{\dd\theta}{\dd u}=\frac{1}{1+u^2}\,.
\]
Hence, this change of variables gives 
\begin{equation}\label{eq:busemann-circle-to-line-energy}
        B(f,f)=4\int_{-\pi/2}^{\pi/2}\int_{-\pi/2}^{\pi/2}f(\theta)f(\phi)\abs{\sin(\theta-\phi)}\,\dd\theta\dd\phi=4\int_{{\mathbb R}}\int_{{\mathbb R}}q(u)q(v)\abs{u-v}\,\dd u\dd v\,.
\end{equation}
Also, with the same change of variables,
\begin{equation}\label{eq:busemann-circle-to-line-mass}
        \int_0^{2\pi}f(\theta)^{2/3}\,\dd\theta=2\int_{{\mathbb R}}q(u)^{2/3}\,\dd u\,.
\end{equation}

Applying \Cref{lem:busemann-line} to $q$ and using \eqref{eq:busemann-circle-to-line-energy} and \eqref{eq:busemann-circle-to-line-mass}, we obtain
\[
        B(f,f)\ge4\cdot\frac{2}{\pi^2}\left(\int_{{\mathbb R}}q(u)^{2/3}\,\dd u\right)^3=\frac1{\pi^2}\left(\int_0^{2\pi}f(\theta)^{2/3}\,\dd\theta\right)^3\,.
\]
This is exactly \eqref{eq:busemann-circle-main}.
\end{proof}

\begin{remark}
The  H\"older equality case in the proof of \Cref{lem:busemann-line} holds when  $y(t)=q(x(t))^{-1/3}$ is proportional to $1/\sqrt{t(1-t)}$.  In the proof of \Cref{lem:busemann-circle-main}, this gives densities coming from centered ellipses.
\end{remark}

\section{Reduction}\label{sec:reduction}
In this section, we reduce the discrete theorems to the continuous ones. Since the arguments for the spherical and planar cases are very similar, we only carry out the reduction of \Cref{thm:main-spherical} to \Cref{thm:continuous-spherical-crossing-main}.  The idea is to replace each vertex by a small smooth cloud of radius $\delta$, and replace each edge by the product of its two endpoint clouds.  For fixed $\delta$, the smoothed vertex density is uniformly close to $1$ once $n$ gets large.  The main point is to show that this smoothing does not create many new crossings. 

We first prove a lemma showing that if a perturbation of a pair of non-crossing edges causes them to cross, then at least one endpoint must be close to the other edge. For illustration, see \Cref{fig:perturbation}.

\begin{lemma} \label{lem:reduction-new-crossing-geometric}
There are absolute positive constants $C$ and $c$ such that the following property holds.  Suppose $0<\alpha<1/10$, let $0<\delta\le c\alpha^2$, and let $x_1,x_2,x_3,x_4 \in {\mathbb S}^2$ be four points satisfying
\[
        \alpha\le d(x_1,x_2),d(x_3,x_4)\le \pi-\alpha\,.
\]
Assume that $[x_1x_2]$ and $[x_3x_4]$ do not cross.  Suppose $y_1,y_2,y_3,y_4\in{\mathbb S}^2$ are such that $d(x_i, y_i) \leq \delta$ and the perturbed segments $[y_1y_2]$ and $[y_3y_4]$ cross, then
\begin{equation}\label{eq:reduction-new-crossing-nearby-endpoint}
        \min\{d(x_1,[x_3x_4]),d(x_2,[x_3x_4]),
              d(x_3,[x_1x_2]),d(x_4,[x_1x_2])\}
        \le C{\delta}/{\alpha}\,.
\end{equation}
\end{lemma}

\begin{figure}[t]
    \centering
\begin{tikzpicture}[
    x=1cm,y=1cm,
    line cap=round,
    line join=round,
    font=\scriptsize,
    spheregrid/.style={draw=black!24, line width=0.45pt},
    xsupport/.style={draw=rmccOriginalBlue!45!black, line width=0.65pt},
    ysupport/.style={draw=rmccPerturbRed!70!black, line width=0.62pt},
    origedge/.style={draw=rmccOriginalBlue, line width=2.35pt,
        preaction={draw=white, line width=4.6pt}},
    pertedge/.style={draw=rmccPerturbRed, line width=1.80pt,
        preaction={draw=white, line width=3.5pt}},
    moveline/.style={draw=black!55, densely dotted, line width=0.60pt, -stealth},
    xpoint/.style={circle, fill=rmccOriginalBlue, draw=white, line width=0.45pt, inner sep=1.85pt},
    ypoint/.style={circle, fill=white, draw=rmccPerturbRed, line width=0.95pt, inner sep=1.75pt},
    labelbox/.style={fill=white, fill opacity=0.88, text opacity=1, inner sep=1.1pt, rounded corners=0.7pt}
]

\definecolor{rmccOriginalBlue}{RGB}{22,70,145}
\definecolor{rmccPerturbRed}{RGB}{190,42,36}

\path[use as bounding box] (-3.22,-3.08) rectangle (3.22,3.08);
\def\R{3.00}

\pgfmathsetmacro{\Ga}{3.00}
\pgfmathsetmacro{\Gb}{0.72}
\pgfmathsetmacro{\Grot}{8}
\pgfmathsetmacro{\Ha}{3.00}
\pgfmathsetmacro{\Hb}{1.372}
\pgfmathsetmacro{\Hrot}{-49.21}

\pgfmathsetmacro{\Ua}{3.00}
\pgfmathsetmacro{\Ub}{1.10}
\pgfmathsetmacro{\Urot}{10}
\pgfmathsetmacro{\Va}{3.00}
\pgfmathsetmacro{\Vb}{1.10}
\pgfmathsetmacro{\Vrot}{-55}

\coordinate (x1) at ({\Ga*cos(118)*cos(\Grot)-\Gb*sin(118)*sin(\Grot)},
                     {\Ga*cos(118)*sin(\Grot)+\Gb*sin(118)*cos(\Grot)});
\coordinate (x2) at ({\Ga*cos(72)*cos(\Grot)-\Gb*sin(72)*sin(\Grot)},
                     {\Ga*cos(72)*sin(\Grot)+\Gb*sin(72)*cos(\Grot)});
\coordinate (x3) at ({\Ha*cos(106.54)*cos(\Hrot)-\Hb*sin(106.54)*sin(\Hrot)},
                     {\Ha*cos(106.54)*sin(\Hrot)+\Hb*sin(106.54)*cos(\Hrot)});
\coordinate (x4) at ({\Ha*cos(69.87)*cos(\Hrot)-\Hb*sin(69.87)*sin(\Hrot)},
                     {\Ha*cos(69.87)*sin(\Hrot)+\Hb*sin(69.87)*cos(\Hrot)});

\coordinate (y1) at ({\Ua*cos(118)*cos(\Urot)-\Ub*sin(118)*sin(\Urot)},
                     {\Ua*cos(118)*sin(\Urot)+\Ub*sin(118)*cos(\Urot)});
\coordinate (y2) at ({\Ua*cos(68)*cos(\Urot)-\Ub*sin(68)*sin(\Urot)},
                     {\Ua*cos(68)*sin(\Urot)+\Ub*sin(68)*cos(\Urot)});
\coordinate (y3) at ({\Va*cos(112)*cos(\Vrot)-\Vb*sin(112)*sin(\Vrot)},
                     {\Va*cos(112)*sin(\Vrot)+\Vb*sin(112)*cos(\Vrot)});
\coordinate (y4) at ({\Va*cos(78)*cos(\Vrot)-\Vb*sin(78)*sin(\Vrot)},
                     {\Va*cos(78)*sin(\Vrot)+\Vb*sin(78)*cos(\Vrot)});
\coordinate (q)  at ({\Ua*cos(76.85)*cos(\Urot)-\Ub*sin(76.85)*sin(\Urot)},
                     {\Ua*cos(76.85)*sin(\Urot)+\Ub*sin(76.85)*cos(\Urot)});

\fill[black!1.5] (0,0) circle (\R);
\begin{scope}
    \clip (0,0) circle (\R);

    \draw[xsupport] plot[variable=\t, domain=0:180, samples=100]
        ({\Ga*cos(\t)*cos(\Grot)-\Gb*sin(\t)*sin(\Grot)},
         {\Ga*cos(\t)*sin(\Grot)+\Gb*sin(\t)*cos(\Grot)});
    \draw[xsupport] plot[variable=\t, domain=0:180, samples=100]
        ({\Ha*cos(\t)*cos(\Hrot)-\Hb*sin(\t)*sin(\Hrot)},
         {\Ha*cos(\t)*sin(\Hrot)+\Hb*sin(\t)*cos(\Hrot)});

    \draw[ysupport] plot[variable=\t, domain=0:180, samples=100]
        ({\Ua*cos(\t)*cos(\Urot)-\Ub*sin(\t)*sin(\Urot)},
         {\Ua*cos(\t)*sin(\Urot)+\Ub*sin(\t)*cos(\Urot)});
    \draw[ysupport] plot[variable=\t, domain=0:180, samples=100]
        ({\Va*cos(\t)*cos(\Vrot)-\Vb*sin(\t)*sin(\Vrot)},
         {\Va*cos(\t)*sin(\Vrot)+\Vb*sin(\t)*cos(\Vrot)});
\end{scope}
\draw[black!55, line width=0.75pt] (0,0) circle (\R);

\draw[origedge] plot[variable=\t, domain=118:72, samples=90]
    ({\Ga*cos(\t)*cos(\Grot)-\Gb*sin(\t)*sin(\Grot)},
     {\Ga*cos(\t)*sin(\Grot)+\Gb*sin(\t)*cos(\Grot)});
\draw[origedge] plot[variable=\t, domain=106.54:69.87, samples=80]
    ({\Ha*cos(\t)*cos(\Hrot)-\Hb*sin(\t)*sin(\Hrot)},
     {\Ha*cos(\t)*sin(\Hrot)+\Hb*sin(\t)*cos(\Hrot)});

\draw[moveline] (x1) -- (y1);
\draw[moveline] (x2) -- (y2);
\draw[moveline] (x3) -- (y3);
\draw[moveline] (x4) -- (y4);

\draw[pertedge] plot[variable=\t, domain=118:68, samples=90]
    ({\Ua*cos(\t)*cos(\Urot)-\Ub*sin(\t)*sin(\Urot)},
     {\Ua*cos(\t)*sin(\Urot)+\Ub*sin(\t)*cos(\Urot)});
\draw[pertedge] plot[variable=\t, domain=112:78, samples=70]
    ({\Va*cos(\t)*cos(\Vrot)-\Vb*sin(\t)*sin(\Vrot)},
     {\Va*cos(\t)*sin(\Vrot)+\Vb*sin(\t)*cos(\Vrot)});

\node[xpoint] at (x1) {};
\node[xpoint] at (x2) {};
\node[xpoint] at (x3) {};
\node[xpoint] at (x4) {};
\node[ypoint] at (y1) {};
\node[ypoint] at (y2) {};
\node[ypoint] at (y3) {};
\node[ypoint] at (y4) {};
\fill[rmccPerturbRed] (q) circle (1.05pt);

\node[labelbox, text=rmccOriginalBlue, anchor=south] at (-1.50, 0) {$x_1$};
\node[labelbox, text=rmccOriginalBlue, anchor=north east] at ( 0.79, 0.70) {$x_2$};
\node[labelbox, text=rmccOriginalBlue, anchor=south west] at ( 0.45, 1.58) {$x_3$};
\node[labelbox, text=rmccOriginalBlue, anchor=west] at ( 1.68, 0.3) {$x_4$};

\node[labelbox, text=rmccPerturbRed, anchor=south] at (-1.56, 0.80) {$y_1$};
\node[labelbox, text=rmccPerturbRed, anchor=south west] at ( 0.94, 1.28) {$y_2$};
\node[labelbox, text=rmccPerturbRed, anchor=east] at ( 0.1, 1.4) {$y_3$};
\node[labelbox, text=rmccPerturbRed, anchor=south west] at ( 0.75, -0.2) {$y_4$};

\end{tikzpicture}
    \caption{Curves $[x_1x_2]$ and $[x_3x_4]$ are non-crossing, but they begin to cross after a small perturbation. Observe that the endpoint $x_2$ is close to the edge $[x_3x_4]$.}
    \label{fig:perturbation}
\end{figure}

\begin{proof}
In the following argument, the absolute constants $c$ and $C$ may vary from line to line. We denote the Hausdorff distance by $d_H$. Recall that the Hausdorff distance between two closed sets is the greatest distance from a point in one set to the nearest point in the other.

We first record the only geometric estimate we need. If $d(x,x')\in[\alpha,\pi-\alpha]$, and  $y,y' \in {\mathbb S}^2$ satisfy $d(x,y),d(x',y')\le\delta$ then
\begin{equation}\label{eq:reduction-segment-hausdorff}
        d_H([xx'],[yy'])\le C{\delta}/{\alpha}\,.
\end{equation}
To see this, compare the supporting great circles.  If
$n=(x\times x')/\abs{x\times x'}$ and $n'=(y\times y')/\abs{y\times y'}$ are unit normals, then $\abs{x\times x'}=\sin d(x,x')\ge\sin\alpha\geq c\alpha$. On the other hand, by linearity, the cross product changes by at most $O(\delta)$.  Thus, after choosing the sign of $n'$,
\[
        \abs{n-n'}\le C{\delta}/{\alpha}\,.
\]
So the supporting great circles $\Gamma(xx')$ and $\Gamma(yy')$ are within Hausdorff distance $O(\delta/\alpha)$.  Rotating $\Gamma(xx')$ onto $\Gamma(yy')$ moves every point by $O(\delta/\alpha)$, and sends the endpoints $x,x'$ to points within $O(\delta/\alpha)$ of $y,y'$.  On a fixed circle, shorter arcs whose endpoints are within $O(\delta/\alpha)$ have Hausdorff distance $O(\delta/\alpha)$, since their lengths remain at most $\pi-\alpha/2$.  This proves \eqref{eq:reduction-segment-hausdorff}.

Now, join each $x_i$ to $y_i$ by the short geodesic path and move the four endpoints continuously.  Since both original edge lengths lie in $[\alpha,\pi-\alpha]$, all intermediate edge lengths remain bounded away from $0$ and $\pi$, if $c$ is chosen sufficiently small.  At time $0$ the two closed segments are disjoint, while at time $1$ they cross.  Let $t$ be the first time at which the two closed moving segments meet.  The first meeting cannot be a transverse interior-interior crossing, since such a crossing would persist for earlier time.  Hence, at time $t$, an endpoint of one moving segment lies on the other moving segment. Suppose at time $t$, edges are $[z_1z_2]$ and $[z_3z_4]$. Without loss of generality, suppose $z_1 \in [z_3z_4]$. Then,
\[
        d(x_1,[x_3x_4])\le d(x_1,z_1)+d_H([z_3z_4],[x_3x_4])\le \delta+C{\delta}/{\alpha}\le C{\delta}/{\alpha}\,.
\]
This concludes the proof.
\end{proof}

Next, we bound the number of edge–vertex pairs for which the vertex lies close to the edge. This bound will later be used to show that only a small number of new crossings are created. For $A\subset {\mathbb S}^2$, write
\[
        N_r(A):=\{x\in{\mathbb S}^2:d(x,A)<r\}\,.
\]
\begin{lemma}\label{lem:reduction-uniform-counts}
Let $P_n=\{p_n(v):v\in V(G_n)\}$, and assume $n^{-1}\sum_{p\in P_n}\delta_p\Rightarrow\mu$.  Then, for any fixed radius $r\in(0,\pi/2)$,
\begin{equation}\label{eq:reduction-uniform-cap-count}
        \sup_{z\in{\mathbb S}^2}\frac1n\#\{p\in P_n: p \in N_r(z)\}\le\frac{1-\cos r}{2}+o_r(1)\,,
\end{equation}
\begin{equation}\label{eq:reduction-uniform-tube-count}
        \sup_{\Gamma}\frac1n\#\{p\in P_n:p\in N_r(\Gamma)\}\le\sin r+o_r(1)\,,
\end{equation}
where the second supremum is over all great circles $\Gamma\subset{\mathbb S}^2$.  Here $o_r(1)\to0$ as $n\to\infty$ with $r$ fixed.
\end{lemma}

\begin{proof}
We prove the tube estimate; the proof for caps is identical.  Suppose the estimate fails for some fixed $r$.  Then there is a subsequence, still denoted by $n$, and great circles $\Gamma_n$ such that
\[
        \frac1n\#(P_n\cap N_r(\Gamma_n))\ge \sin r+\varepsilon
\]
for some $\varepsilon>0$.  The space of great circles is compact, so after passing to a further subsequence we may assume $\Gamma_n\to\Gamma$.  For every $\gamma>0$, and all large $n$, we have $ N_r(\Gamma_n)\subset \overline{N_{r+\gamma}(\Gamma)}$. By weak convergence,
\[
        \limsup_{n\to\infty}\frac1n\#(P_n\cap N_r(\Gamma_n))\le\mu\bigl(\overline{N_{r+\gamma}(\Gamma)}\bigr)=\sin(r+\gamma)\,.
\]
Letting $\gamma\downarrow0$ gives a contradiction.  The cap estimate follows from  $\mu(N_r(x)) = (1-\cos r)/2$ and analogous argument.
\end{proof}

We can finally carry out the precise reduction, using the ingredients above.
\begin{proposition}[Reduction]\label{prop:discrete-continuous-reduction}
Let ${\mathcal D}_n=(G_n,p_n)$ be a good uniform sequence of spherical geodesic drawings. There are positive constants $c$ and $C$, such that for every fixed $\delta \in (0,c)$, there are smooth functions $\rho_{n,\delta}:{\mathbb S}^2\to[0,\infty)$ and $ w_{n,\delta}:{\mathbb S}^2\times{\mathbb S}^2\to[0,\infty)$ with the following properties.
\begin{itemize}
    \item[(i)]  $\rho_{n, \delta}$ is a probability density on ${\mathbb S}^2$ that uniformly converges to the uniform density $1$.  
    \begin{equation}\label{eq:reduction-rho-close}
         \int_{{\mathbb S}^2}\rho_{n,\delta}(x)\,\mu(\dd x)=1
        \quad \text{ and } \quad \norm{\rho_{n,\delta}-1}_{L^\infty(\mu)}\to0\,.
    \end{equation}
    \item[(ii)] The edge density $w_{n,\delta}$ is symmetric, $w_{n,\delta}(x,y)\leq \rho_{n,\delta}(x)\rho_{n,\delta}(y)$, and 
    \begin{equation}
        e(w_{n,\delta}) = 2|E(G_n)|/n^2\,.
    \end{equation}
    \item[(iii)] For fixed $\delta$ and $n\to \infty$, the crossing density is bounded: 
    \begin{equation}\label{eq:reduction-crossing-bound}
    \operatorname{Cr}(w_{n,\delta}) \leq 
        8\,\frac{\operatorname{cr}({\mathcal D}_n)}{n^4}+C\sqrt{\delta}+o_\delta(1)\,.
\end{equation}
\end{itemize}
\end{proposition}

\begin{proof}
Fix $\delta \in (0,c)$.  Choose a smooth nonnegative radial kernel $K:{\mathbb S}^2\times{\mathbb S}^2\to[0,\infty)$ such that $K(x,y)$ depends only on $d(x,y)$, is supported on $d(x,y)<\delta$, and is normalized so that
\[
        \int_{{\mathbb S}^2}K(x,y)\,\mu(\dd x)=1\qquad\text{for every }y\in{\mathbb S}^2\,.
\]
For each vertex $v\in V(G_n)$, let us identify $v$ with $p_n(v) \in {\mathbb S}^2$. Write  $K_v(x)=K(x,v)$ and define
\begin{equation}\label{eq:reduction-rho-definition-simple}
        \rho_{n,\delta}(x):=\frac1n\sum_{v\in V(G_n)}K_v(x)=\int_{{\mathbb S}^2}K(x,y)\,\nu_n(\dd y),\quad \text{ where }\nu_n=\frac1n\sum_{v\in V(G_n)}\delta_{v}\,.
\end{equation}
This shows that $\rho_{n,\delta}$ is a probability density.  Since $\nu_n\Rightarrow\mu$ and the family
$\{K(x,\cdot):x\in{\mathbb S}^2\}$ is compact in $C({\mathbb S}^2)$, the convergence of these integrals is uniform in $x \in {\mathbb S}^2$.  Hence
\[
        \rho_{n,\delta}(x)\to\int_{{\mathbb S}^2}K(x,y)\,\mu(\dd y)=1\qquad\text{uniformly in }x\,,
\]
which proves \eqref{eq:reduction-rho-close}. Let $\vec E(G_n)$ denote the set of the two orientations of each edge of $G_n$, and define
\begin{equation}\label{eq:reduction-w-definition-simple}
         w_{n,\delta}(x,y):=\frac1{n^2}\sum_{(u,v)\in\vec E(G_n)}K_u(x)K_v(y)\,.
\end{equation}
This function is symmetric.  Moreover, from here we see $  w_{n,\delta}(x,y) \leq \rho_{n,\delta}(x)\rho_{n,\delta}(y)$. Also, integrating above, we obtain $e(w_{n,\delta}) = 2|E(G_n)|/n^2$. It now remains to compare crossings. Expanding $w_{n, \delta }$, we observe that
\begin{equation}\label{eq:reduction-crossing-expanded-simple}
            \operatorname{Cr}(w_{n,\delta})=\frac1{n^4}\sum_{\substack{(a,b)\in\vec E(G_n) \\ (u,v)\in\vec E(G_n)}}\Pr\bigl([U_uU_v]\text{ crosses }[U_aU_b]\bigr)\,,
\end{equation}
where, for each occurrence of a vertex $v$, the random point $U_v$ is sampled independently from the density $K_v\,\mu$.  Thus every sampled endpoint lies within spherical distance $\delta$ of the original endpoint. We decompose the ordered pairs in \eqref{eq:reduction-crossing-expanded-simple}.  Ordered pairs of oriented edges sharing a vertex clearly contribute only $O(1/n)$.   If two independent edges cross in the original drawing, their unordered pair contributes at most eight ordered and oriented terms to \eqref{eq:reduction-crossing-expanded-simple}.  Hence all originally crossing independent pairs contribute at most
\[
        8\,\frac{\operatorname{cr}({\mathcal D}_n)}{n^4}.
\]
It remains to control independent pairs of edges that do not cross in the original drawing but may cross after the endpoint perturbations.

Call an unordered independent non-crossing pair of edges $\{x_1x_2,\,x_3x_4\}$ dangerous if there exist points $y_i \in {\mathbb S}^2$ with $d(x_i, y_i) \leq \delta$ such that  $[y_1y_2]$ crosses $[y_3y_4]$.  It suffices to show that the number of dangerous pairs of edges  is at most $[C\sqrt{\delta} +o_\delta(1)]n^4$. 

Let $c_0,C_0$ be the constants from \Cref{lem:reduction-new-crossing-geometric}.  Choose a large constant $A\ge 1/\sqrt{c_0}$, and set $\alpha=A\sqrt{\delta}.$ If $c$ is small enough, we may assume $0<\alpha<1/10$.  Then $\delta\le c_0\alpha^2$, so \Cref{lem:reduction-new-crossing-geometric} applies to every pair of edges whose lengths both lie in $[\alpha,\pi-\alpha]$.

First discard the unstable edges, namely those with length outside $[\alpha,\pi-\alpha]$.  For each endpoint $p$, the other endpoint of such an edge must lie in the $\alpha$-cap around $p$ or in the $\alpha$-cap around $-p$.  By \Cref{lem:reduction-uniform-counts} \eqref{eq:reduction-uniform-cap-count}, the number of unstable edges is at most $C[\alpha^2+o_\alpha(1)]n^2$. Hence, the number of edge pairs containing at least one unstable edge is at most $C[\alpha^2+o_\alpha(1)]n^4.$

Consider now a dangerous pair $\{x_1x_2, \, x_3x_4\}$ in which both edges are stable.  By \Cref{lem:reduction-new-crossing-geometric}, one endpoint of one edge is within $R = C_0 \delta / \alpha$ of the other edge segment.  By decreasing $c$ again, assume $R<\pi/2$.  Fix a stable edge $x_1x_2$.  If $\Gamma({x_1x_2})$ is its supporting great circle, then
\[
        N_R([x_1x_2])\subseteq N_R(\Gamma({x_1x_2}))\,.
\]
Thus \Cref{lem:reduction-uniform-counts} gives, uniformly in $x_1x_2$,
\[
        \#\{v\in V(G_n):v\in N_R([x_1x_2])\}
        \le
        \bigl(\sin R+o_R(1)\bigr)n
        \le
        \bigl(CR+o_R(1)\bigr)n\,.
\]
After choosing this near endpoint of the second edge, there are at most $n$ choices for its other endpoint.  Summing over the at most $n^2$ choices of the base edge gives at most $[CR + o_R(1)]n^4$ stable dangerous pairs. Combining the bounds above and using $\alpha=O(\sqrt{\delta})$ and $R=O(\sqrt{\delta})$, we get that the number of dangerous pairs is at most
\[        \bigl(C\alpha^2+CR+o_\delta(1)\bigr)n^4
        \le
        \bigl(C\sqrt{\delta}+o_\delta(1)\bigr)n^4\,.
\]
This concludes the proof.
\end{proof}

We also mention an easy corollary of \Cref{thm:continuous-spherical-crossing-main}.

\begin{corollary}\label{cor:near-uniform}
Let $\rho$ be a probability density on ${\mathbb S}^2$ with $\norm{\rho-1}_{L^\infty(\mu)}\le\eta$. Let $w:{\mathbb S}^2\times{\mathbb S}^2\to[0,\infty)$ be a symmetric function such that $w(x,y)\leq \rho(x)\rho(y)$. If $t \in [0,\pi]$ satisfies $\cos t = 1-2e(w) / (1+\eta)^2$, then 
\begin{equation}\label{eq:near-uniform}
    \operatorname{Cr}( w) \geq (1+\eta)^4\frac{(\sin t- t\cos t)^2}{8\pi^2}\,.
\end{equation}
\end{corollary}

\begin{proof}
Since $w(x,y)\le \rho(x)\rho(y)$ and $\rho\le1+\eta$, the function $\tilde w = w / (1+\eta)^2$ takes values in $[0,1]$.  Hence $\tilde w$ is an admissible continuous edge density for \Cref{thm:continuous-spherical-crossing-main}.  Its density is $\tilde e = e(w)/(1+\eta)^2$, so its crossing functional is $\operatorname{Cr}(\tilde w) = \operatorname{Cr}(w) / (1+\eta)^4$.  Applying \Cref{thm:continuous-spherical-crossing-main} to $\tilde w$ gives \eqref{eq:near-uniform}.  
\end{proof}

We are now ready to combine \Cref{cor:near-uniform} and \Cref{prop:discrete-continuous-reduction} to prove \Cref{thm:main-spherical}.

\begin{proof}[Proof of \Cref{thm:main-spherical}]
Denote the ordered density of $G_n$ by $e_n$. Fix $\delta \in (0,c)$ with $c$ as in \Cref{prop:discrete-continuous-reduction}, and apply the proposition to ${\mathcal D}_n$.  Put $\eta_{n,\delta}:= \norm{\rho_{n,\delta}-1}_{L^\infty(\mu)}$ and  $t_{n,\delta}\in[0,\pi]$ with
\begin{equation*}
        \cos t_{n,\delta}=1-\frac{2e_n}{(1+\eta_{n,\delta})^2}\,.
\end{equation*}
By \Cref{cor:near-uniform} applied with $\rho=\rho_{n,\delta}$ and $w=w_{n,\delta}$ and by \Cref{prop:discrete-continuous-reduction}, as $n \to \infty$
\begin{equation}\label{eq:main-spherical-proof-lower-continuous}
        8\frac{\operatorname{cr}({\mathcal D}_n)}{n^4}+C\sqrt{\delta}+o_\delta(1) \geq \operatorname{Cr}(w_{n,\delta})\ge(1+\eta_{n,\delta})^4\frac{(\sin t_{n,\delta}-t_{n,\delta}\cos t_{n,\delta})^2}{8\pi^2}\,.
\end{equation}
Since $e_n\to e$ and $\eta_{n,\delta}\to0$, we have $t_{n,\delta}\to t$, where $\cos t=1-2e$.  Therefore, for all $\delta\in(0,c)$,
\begin{equation*}
        \liminf_{n\to\infty}8\frac{\operatorname{cr}({\mathcal D}_n)}{n^4}\ge\frac{(\sin t-t\cos t)^2}{8\pi^2}-C\sqrt{\delta}\,.
\end{equation*}
Letting $\delta\downarrow0$, we conclude \eqref{eq:main-spherical-sharp-density}.
\end{proof}

\section{Nonuniform counterexample}\label{sec:nonuniform-counterexample}

In this section we give the proof of \Cref{thm:main_counterexample}. The construction will preserve the distance threshold edge function but allow the vertex density to vary periodically. The key is to find a small perturbation of the uniform density for which the ratio of the number of crossings over the number of edges cubed decreases.

\subsection{Periodic densities}

Let $R=[0,\ell]^2$ for some $\ell>0$, and let $\rho:\mathbb R^2\to[0,\infty)$ be a bounded measurable function which is periodic under translations by $\ell\mathbb Z^2$ and satisfies
\[
        \frac1{|R|}\int_R\rho(x)\,\dd x=1\,.
\]
Define edge and crossing intensities per unit area by
\begin{align}
        E(\rho)&:=\frac1{|R|}\int_R\int_{\mathbb R^2}\rho(x)\rho(y)\II{\{\norm{x-y}\le1\}}\,\dd y\,\dd x\,,\label{eq:nonuniform-periodic-edge}\\
        \operatorname{Cr}(\rho)&:=\frac1{|R|}\int_R\int_{(\mathbb R^2)^3}\rho(x)\rho(y)\rho(z)\rho(w)\II{\{\norm{x-y}\le1\}}\II{\{\norm{z-w}\le1\}}\notag\\[-2pt]
        &\hspace{7em}\cdot\II{\{[xy]\text{ crosses }[zw]\}}\,\dd y\,\dd z\,\dd w\,\dd x\,.
        \label{eq:nonuniform-periodic-crossing}
\end{align}

First, let us compute the values at the uniform density. For the edges, we have that 
\begin{equation*}
    E(1) = \int_{\mathbb R^2} \II{\{\norm{y}\le1\}} \,\dd y = \pi\,.
\end{equation*}
Similarly, note that if we fix $u = y-x$ and $v = w-z$, then the set of points $z$ for which segment $[x,x+u]$ crosses segment $[z, z+v]$ forms a parallelogram of area $|\det(u,v)|$. Thus,
\begin{equation*}
     \operatorname{Cr}(1) = \int_{\norm{u} \leq 1}\int_{\norm v\le1} |\det(u,v)| \, \dd v \, \dd u = \frac{4}{3} \int_{\norm{u} \leq 1} \norm{u} \, \dd u = \frac{8\pi}{9}\,.
\end{equation*}

Note that $\operatorname{Cr}(1) / E(1)^3 = 8 / 9\pi^2$. The following observation lets us work with periodic densities instead of finite drawings.

\begin{lemma}\label{lem:nonuniform-periodic-transfer}
For any periodic density $\rho$, the midrange crossing constant satisfies
\[
        \kappa\le\frac{\operatorname{Cr}(\rho)}{E(\rho)^3}\,.
\]
\end{lemma}

\begin{proof}
Let $L$ be a large multiple multiple of $\ell$, put $\Omega_L=[0,L]^2$, and let $\lambda_L=L^{-2}\mathcal L^2|_{\Omega_L}$ be the uniform measure on $\Omega_L$. Note that periodicity implies that $\rho\lambda_L$ is a probability measure. Let
\[
        w_L(x,y):=\rho(x)\rho(y)\II{\{\norm{x-y}\le1\}}\,.
\]
Define $e_{\Omega_L}(w_L)$ and $\operatorname{Cr}_{\Omega_L}(w_L)$ as in the \Cref{subsec:main-planar}. Then, we have have that 
\begin{equation}\label{eq:nonuniform-boundary-errors}
\begin{aligned}
        L^2e_{\Omega_L}(w_L)&=E(\rho)+O_\rho(L^{-1})\,,\\
        L^6\operatorname{Cr}_{\Omega_L}(w_L)
        &=\operatorname{Cr}(\rho)+O_\rho(L^{-1})\,.
\end{aligned}
\end{equation}
This is because only the configurations whose first endpoint lies within distance two of the boundary of $\Omega_L$ contribute to the error.

Fix a sufficiently large $L$, sample $n$ vertices independently from $\rho\lambda_L$, and join pairs at distance at most one to form a graph $G_n$ with the inherited drawing $\mathcal{D}_n$. The drawing is good almost surely, since the sampling measure is absolutely continuous. Note that 
\begin{equation*}
    \mathbb E\left[\frac{2E(G_n)}{n^2}\right] = e_{\Omega_L}(w_L) \quad \text{ and } \quad \mathbb E \left[\frac{8\operatorname{cr}(\mathcal D_n)}{n^4}\right] = \operatorname{Cr}_{\Omega_L}(w_L)\,.
\end{equation*}
Furthermore, note that variances for both quantities are $O(1/n)$. Therefore, by Chebyshev's inequality, we can choose realizations $G_n$ and $\mathcal D_n$ such that 
\[
        \frac{2|E(G_n)|}{n^2}\to e_{\Omega_L}(w_L)\,,
        \qquad\frac{8\operatorname{cr}(\mathcal D_n)}{n^4}\to\operatorname{Cr}_{\Omega_L}(w_L)\,.
\]

Taking $n$ disjoint, sufficiently separated translated copies of $\mathcal D_n$ gives a drawing with $N_n=n^2$ vertices and $M_n=n|E(G_n)|=\Theta(n^3)$ edges. Its crossing ratio is unchanged. Since $N_n\ll M_n\ll N_n^2$ and the crossing number of a graph is at most the number of crossings in any of its good drawings,
\Cref{thm:midrange_exists} gives
\[
        \kappa\le\frac{\operatorname{Cr}_{\Omega_L}(w_L)}{e_{\Omega_L}(w_L)^3}\,.
\]
Letting $L\to\infty$ and using
\eqref{eq:nonuniform-boundary-errors} proves the claim.
\end{proof}

\subsection{Pertubations of uniform density}

By \Cref{lem:nonuniform-periodic-transfer}, it is enough to find a periodic density for which $\operatorname{Cr}(\rho)/E(\rho)^3$ is smaller than its value at the uniform density. We consider a small perturbation
$\rho=1+\varepsilon f$, where $f$ is bounded and $\ell$-periodic, with 
\[\int_R f(x)\, \dd x = 0\]
For sufficiently small $\varepsilon>0$, this is a nonnegative density of mean one.  We now study how the edge and crossing intensity change per unit area.

\begin{lemma}\label{lem:nonuniform-edge-variation}
We have the exact identity
\begin{equation}\label{eq:nonuniform-edge-variation}
        E(1+\varepsilon f)=E(1) + \frac{\varepsilon^2}{|R|}\int_R\int_{\mathbb R^2}f(x)f(y)\II{\{\norm{x-y}\le1\}}\,\dd y\,\dd x\,.
\end{equation}
\end{lemma}

\begin{proof}
Writing $v = y-x$ and expanding gives us 
\[
\begin{aligned}
         E(1+\varepsilon f)-E(1)
        =&\frac{\varepsilon}{|R|}\int_{\norm v\le1}
        \underbrace{\int_R\bigl[f(x)+f(x+v)\bigr]\,\dd x}_{=0}\,\dd v\\
        &+\frac{\varepsilon^2}{|R|}
        \int_{\norm v\le1}\int_R f(x)f(x+v)\,\dd x\,\dd v\,.
\end{aligned}
\]
Changing back to $y=x+v$ in the remaining term proves the identity.
\end{proof}

For the crossing intensity, it is convenient to define 
\begin{equation}\label{eq:nonuniform-marginal-kernel}
        H(v):= \int_{\norm u\le1}\int_{\norm z\le1}\II{\{[0u]\text{ crosses }[v,v+z]\}}\,\dd z\,\dd u\,.
\end{equation}

\begin{lemma}\label{lem:nonuniform-crossing-variation}
As $\varepsilon\to0$ with $f$ fixed, we have
\begin{equation}\label{eq:nonuniform-crossing-variation}
\begin{aligned}
        \operatorname{Cr}(1+\varepsilon f)-\operatorname{Cr}(1)
        =&\frac{\varepsilon^2}{|R|}\int_R\int_{\mathbb R^2}
        f(x)f(y)\left(\frac83\norm{x-y}\II{\{\norm{x-y}\le1\}}+4H(y-x)\right)\,\dd y\,\dd x\\
        &+O_f(\varepsilon^3)\,.
\end{aligned}
\end{equation}
\end{lemma}

\begin{proof}
Expand the product $\rho(x)\rho(y)\rho(z)\rho(w) =(1+\varepsilon f(x))(1+\varepsilon f(y)) (1+\varepsilon f(z))(1+\varepsilon f(w))$. Note that the linear term vanishes as $\int_R f = 0$. Moreover, since $f$ is bounded, cubic and quartic terms contribute $O_f(\varepsilon^3).$ Finally, the quadratic terms are
\[
        \varepsilon^2\Bigl[\underbrace{f(x)f(y)+f(z)f(w)}_{\text{two same-edge terms}}+\underbrace{(f(x)+f(y))(f(z)+f(w))}_{\text{four mixed terms}}\Bigr]\,.
\]
For a fixed edge $[xy]$, we have that
\[
        \int_{\norm u\le1}\int_{\mathbb R^2}\II{\{[xy]\text{ crosses }[z,z+u]\}}\,\dd z\,\dd u =\int_{\norm u\le1}|\det(y-x,u)|\,\dd u=\frac43\norm{y-x}\,.
\]
Therefore, the two same-edge terms contribute 
\[
2\times\frac{\varepsilon^2}{|R|}\int_R\int_{\mathbb R^2}f(x)f(y)\times \frac43\norm{x-y}\II{\{\norm{x-y}\le1\}} \, \dd y \, \dd x\,.
\]
As for the mixed terms, if we change the variables and write the edges as $[x, x+u]$ and $[y, y+v]$, then by \eqref{eq:nonuniform-marginal-kernel}, the contribution is precisely 
\begin{equation*}
    4 \times \frac{\varepsilon^2}{|R|}\int_R\int_{\mathbb R^2}f(x)f(y)\times H(y-x)\, \dd y \, \dd x\,.
\end{equation*}
Hence, we are done.
\end{proof}

In the next subsection, we will need to bound the contribution coming from the $H$ term. To this end, we establish the following regularity lemma.

\begin{lemma}\label{lem:nonuniform-kernel-regularity}
There exists a continuously differentiable Lipschitz function $b:[0,\infty)\to\mathbb R$ such that
\[H(v) = b(\norm v)\,,\]
for $H$ defined as in \eqref{eq:nonuniform-marginal-kernel} and $v \neq 0.$
\end{lemma}

\begin{proof}
Rotation invariance shows that $H$ is radial. Set $b(r):=H(re_1)$ for $r>0$. For a crossing above the horizontal axis, let $0<\theta<\phi<\pi$ be the directions from $0$ and $re_1$ to the intersection, and put
\[
        a=\frac{\sin\phi}{\sin(\phi-\theta)}\,,\qquad d=\frac{\sin\theta}{\sin(\phi-\theta)}\,.
\]
By rule of sines, the distances to the intersection are $ra$ and $rd$. Polar coordinates give
\begin{equation}\label{eq:nonuniform-angular-marginal}
\begin{aligned}
        b(r)
        &=2\int_{0<\theta<\phi<\pi}
        \left(\int_0^1s\II{\{s>ra\}}\,\dd s\right)\left(\int_0^1t\II{\{t>rd\}}\,\dd t\right)
        \,\dd\theta\,\dd\phi\\
        &=\frac12\int_{0<\theta<\phi<\pi}(1-r^2a^2)_+(1-r^2d^2)_+\,\dd\theta\,\dd\phi\,,
\end{aligned}
\end{equation}
where $s_+=\max\{s,0\}$. First, note that when $r > r_0 > 0$, the integrant is either $0$, or $ra, rd < 1$. This means that the derivative of the integrant is uniformly bounded, that is,
\[
        \left|\partial_r\bigl[(1-r^2a^2)_+(1-r^2d^2)_+\bigr]\right|\le 2 r a^2 +  2rd^2 + 4r^3a^2d^2 \leq \frac{8}{r_0}\,.
\]

By dominated convergence theorem, $b$ is continuously differentiable on $(0, \infty)$ and
\[
b'(r)=-2r\int_{0<\theta<\phi<\pi} a^2\II{\{ra<1\}}(1-r^2d^2)_+\,\dd\theta\,\dd\phi\,.
\]
Now, it suffices to bound $|b'(r)|$. Fix $\phi$ and substitute
\[
        s=rd=\frac{r\sin\theta}{\sin(\phi-\theta)}\,,\qquad\frac{\dd s}{\dd\theta} =\frac{r\sin\phi}{\sin^2(\phi-\theta)}=\frac{r a^2}{\sin \phi}\,.
\]
As $\theta$ runs from $0$ to $\phi$, $s$ runs from $0$ to $\infty$, and $(ra)^2=s^2+2rs\cos\phi+r^2$. Hence,
\[
        b'(r)=-2\int_0^\pi\sin\phi\int_0^\infty(1-s^2)_+\II{\{s^2+2rs\cos\phi+r^2<1\}}\,\dd s\,\dd\phi\,.
\]
In particular, we can bound the derivative uniformly
\[
        |b'(r)|\le2\int_0^\pi\sin\phi\,\dd\phi\int_0^1(1-s^2)\,\dd s=\frac83\,,
        \qquad\lim_{r\downarrow0}b'(r)=-\frac83\,,
\]
where the limit follows by dominated convergence. This means that the function $b$ is Lipschitz and $b(0+)$ is well defined. Set $b(0) = b(0+)$ to conclude the proof.
\end{proof}

Our goal is now to choose $f$ and
$\varepsilon>0$ so that
\[
        \frac{\operatorname{Cr}(1+\varepsilon f)}{E(1+\varepsilon f)^3}<\frac{\operatorname{Cr}(1)}{E(1)^3}=\frac8{9\pi^2}\,.
\]
For simplicity, write $\Delta_{\operatorname{C}} := \operatorname{Cr}(1 + \varepsilon f) - \operatorname{Cr}(1)$ and similarly define $\Delta_E := \operatorname{E}(1 + \varepsilon f) - \operatorname{E}(1) $. Since $\Delta_{\operatorname{C}}$ and $\Delta_{\operatorname{E}}$ are $O(\varepsilon^2)$, expanding the denominator above gives
\begin{equation}\label{eq:nonuniform-ratio-expansion}
        \frac{\operatorname{Cr}(1+\varepsilon f)}{E(1+\varepsilon f)^3}-\frac8{9\pi^2}
        = \frac{\Delta_{\operatorname{C}} + 8\pi/9 }{(\Delta_{\operatorname{E}} + \pi)^3} - \frac{8}{9\pi^2}=\frac1{\pi^3}\left(\Delta_{\operatorname{C}}-\frac83\Delta_E\right)+O_f(\varepsilon^4)\,.
\end{equation}
Now, using \eqref{eq:nonuniform-edge-variation} and \eqref{eq:nonuniform-crossing-variation}, we see that
\begin{equation*}
        \Delta_{\operatorname{C}}-\frac83\Delta_E = \frac{\varepsilon^2}{|R|}\int_R\int_{\mathbb R^2}
        f(x)g(y)\left[\frac{8}{3}(\norm{y-x} - 1)\II{\{\norm{y-x}\le1\}} + 4 H(y-x)\right]\,\dd y\,\dd x + O(\varepsilon^3)\,.
\end{equation*}

Let us write
\begin{equation}\label{eq:nonuniform-cutoff-kernel}
        K(v):=(\norm v-1)\II{\{\norm v\le1\}}\,.
\end{equation}
Also, for a bounded, measurable, compactly supported even kernel $T:\mathbb R^2\to\mathbb R$ and bounded periodic functions $f,g$, write
\[
        Q_T(f,g):=\frac1{|R|}\int_R\int_{\mathbb R^2}f(x)g(y)T(y-x)\,\dd y\,\dd x\,.
\]
Then, we can rewrite the above difference as 
\begin{equation}\label{eq:nonuniform-quadratic-expansion}
        \Delta_{\operatorname{C}}-\frac83\Delta_E=\varepsilon^2\left(\frac83Q_K(f,f)+4Q_H(f,f)\right)+O_f(\varepsilon^3)\,.
\end{equation}
Therefore, it suffices to find a bounded periodic $f$ with $\int_R f=0$ such that
\begin{equation}\label{eq:nonuniform-negative-variation}
        \frac83Q_K(f,f)+4Q_H(f,f)<0\,.
\end{equation}
Once such an $f$ is fixed, equations \eqref{eq:nonuniform-ratio-expansion} and \eqref{eq:nonuniform-quadratic-expansion} give the desired strict inequality for all sufficiently small $\varepsilon>0$.

\subsection{Main construction}

We now take $R=[0,5]^2$. The following proposition gives the required perturbation explicitly.

\begin{proposition}\label{prop:nonuniform-counterexample}
For $\delta>0$ define two thin rectangles within $R = [0,5]^2$ by
\[
        A:=[2,2+\delta]\times[2,2+\sqrt{\delta}]\,,\qquad B:=A+(1+6\delta)e_1\,.
\]
For $U\in\{A,B\}$, let $f_U$ be the $5\mathbb Z^2$-periodic
extension of
\[
        \II{\{x\in U-4\delta e_1\}}-2\II{\{x\in U\}}+\II{\{x\in U+4\delta e_1\}}\,,
\]
and set $\rho=1+\varepsilon(f_A+f_B)$ (see Figure~\ref{fig:nonuniform-perturbation}). Then, for every sufficiently small $\delta>0$ and then every sufficiently
small $\varepsilon>0$, we have that
\begin{equation}\label{eq:nonuniform-periodic-improvement}
        \frac{\operatorname{Cr}(\rho)}{E(\rho)^3}<\frac8{9\pi^2}\,.
\end{equation}
\end{proposition}

\begin{figure}[tb]
\centering
\begin{tikzpicture}[x=1cm,y=1cm,font=\small,
    positive/.style={draw=blue!80!black,fill=blue!12,line width=.5pt},
    negative/.style={draw=red!80!black,fill=red!10,line width=.5pt},
    dimension/.style={<->,>=stealth,line width=.4pt},
    guide/.style={draw=black!35,line width=.35pt}]

\node[above] at (1.5,3.65) {$R=[0,5]^2$};
\draw[line width=.55pt] (0,.65) rectangle (3,3.65);

\begin{scope}[shift={(-.084,.89)},x=.6cm,y=.6cm]
  \foreach \base in {2,3.24}{
    \path[positive,line width=.25pt]
      ({\base-.16},2) rectangle ({\base-.12},2.2);
    \path[negative,line width=.25pt]
      (\base,2) rectangle ({\base+.04},2.2);
    \path[positive,line width=.25pt]
      ({\base+.16},2) rectangle ({\base+.20},2.2);
  }
\end{scope}

\draw[guide,dashed] (.95,2.02) rectangle (2.05,2.28);
\draw[guide,dashed] (2.05,2.28) -- (4.4,2.5);
\draw[guide,dashed] (2.05,2.02) -- (4.4,.8);

\foreach \x/\weight/\name/\style in {
    5.05/{+1}/{A-4\delta e_1}/positive,
    5.9/{-2}/{A}/negative,
    6.75/{+1}/{A+4\delta e_1}/positive,
    9.45/{+1}/{B-4\delta e_1}/positive,
    10.3/{-2}/{B}/negative,
    11.15/{+1}/{B+4\delta e_1}/positive}{
  \path[\style] ({\x-.24},.8) rectangle ({\x+.24},2.5);
  \node[font=\scriptsize] at (\x,1.65) {$\weight$};
  \ifnum\weight<0\relax
    \node[below,inner sep=4pt,font=\footnotesize]
      at (\x,.8) {$\name$};
  \else
    \node[above,inner sep=4pt,font=\footnotesize]
      at (\x,2.5) {$\name$};
  \fi
}

\draw[guide] (5.9,2.95) -- (5.9,3.52);
\draw[guide] (10.3,2.95) -- (10.3,3.52);
\draw[dimension] (5.9,3.4)
  -- node[above] {$1+6\delta$} (10.3,3.4);

\draw[dimension] (6.99,1.65)
  -- node[above,font=\footnotesize] {$1-3\delta$} (9.21,1.65);

\draw[guide] (4.81,.8) -- (4.81,.37);
\draw[guide] (5.29,.8) -- (5.29,.37);
\draw[dimension] (4.81,.46)
  -- node[below] {$\delta$} (5.29,.46);

\draw[guide] (11.39,.8) -- (11.95,.8);
\draw[guide] (11.39,2.5) -- (11.95,2.5);
\draw[dimension] (11.8,.8)
  -- node[right] {$\sqrt\delta$} (11.8,2.5);

\end{tikzpicture}
\caption{Pertubation in one periodic cell.}
\label{fig:nonuniform-perturbation}
\end{figure}

\begin{proof}[Proof of \Cref{prop:nonuniform-counterexample}]
Let $f:=f_A+f_B$. For all sufficiently small $\delta$, the six rectangles are disjoint and lie in the interior of $R$. Moreover, since $f$ is bounded and $\int_R f=0$, it remains only to verify \eqref{eq:nonuniform-negative-variation} for all sufficiently small $\delta$. That is, we must show that
\[ \frac83Q_K(f,f)+4Q_H(f,f)<0.\]
Note that we can use bilinearity and symmetry to see that it suffices to show
\[ \frac{16}{3}Q_K(f_A,f_A) + \frac{16}{3}Q_K(f_A,f_B)+8Q_H(f_A,f_A) + 8Q_H(f_A,f_B) < 0\,.\]

To bound expressions of the form $Q_T(f_U, f_V)$, we write them in a more convenient form as follows. Sample $a\in U$ and $b\in V$ independently and uniformly at random. Independently of these points, let $\sigma,\tau$ be
independent random variables with the same law:
\[
        \mathbb P(\sigma=-1)=\mathbb P(\sigma=1)=\frac14\,,
        \qquad \mathbb P(\sigma=0)=\frac12\,.
\]
Then, observe that we can write 
\begin{equation}\label{eq:nonuniform-random-shifts}
        Q_T(f_U,f_V)=\frac{16|U||V|}{|R|}\mathbb E\left[(-1)^{\sigma+\tau}T\bigl(b-a+(\tau-\sigma)\times 4\delta e_1\bigr)\right]\,.
\end{equation}
In the remainder of the proof, we will show that there exists some absolute constant $c > 0$ such that 
\begin{equation*}
    Q_K(f_A,f_B) \leq -c \delta^4 \quad  \text{ while } \quad  Q_K(f_A,f_A) ,Q_H(f_A,f_B) ,Q_H(f_A,f_A)  = o(\delta^4)\,.
\end{equation*}

We first bound the negligible terms. Let's start with $Q_H(f_A,f_B)$. Observe that when $a \in A$ and $b \in B$, then the vector $b-a + (\tau - \sigma) \times 4\delta e_1$ is in the neighborhood of $e_1$ where $H$ is continuously differentiable. Hence, if $z = b-a -\sigma \times 4\delta e_1$, then
\begin{equation*}
    T(z + \tau \times 4\delta  e_1) = T(z) + \tau \times 4\delta \,  \partial_xT(z) + o(\delta)\,.
\end{equation*}
Taking expectation over $\tau$, since $\mathbb[\tau (-1)^\tau] = 0 $ we get that
\begin{equation*}
    \mathbb E_\tau [ (-1)^\tau T(z + \tau \times 4\delta  e_1) ] = o(\delta) \implies Q_H(f_A,f_B)= |A||B| \times o(\delta) = o(\delta^4)\,.
\end{equation*}

We now use the same idea to bound $Q_H(f_A,f_A)$ and $Q_K(f_A,f_A)$. The relevant differences $A-A$ now lie in the neighborhood of $0$. However, the kernels are no longer continuously differentiable near the origin, but they are still Lipschitz, which is what we will use.

By \Cref{lem:nonuniform-kernel-regularity} and \eqref{eq:nonuniform-cutoff-kernel}, both kernels $T \in \{K,H\}$ satisfy $T(v) = t(\norm{v})$ for some Lipschitz function $t$ in the neighborhood of the origin. Note that by the Lipschitz property, we have that
\[\bigl|t\bigl(\sqrt{x^2 + y^2}\bigr) - t\bigl(y\bigr)\bigr| \lesssim \bigl| \sqrt{x^2 - y^2} - y\bigr| \lesssim  \min \left\{x, \frac{x^2}{y}\right\}\,.\]
Therefore, we will consider two cases depending on the vertical separation $a-b$. First, observe that 
\begin{equation*}
    \mathbb P\{|\langle a-b, e_2\rangle| \leq \delta^{3/4}\} \leq O(\delta^{1/4})\,.
\end{equation*}
This means that the contribution from that event is bounded by
\begin{equation*}
    \left|\mathbb E\bigl[(-1)^{\sigma+\tau}
        T\bigl(b-a+(\tau-\sigma)\times 4\delta e_1\bigr) \II\{|\langle a-b, e_2\rangle| \leq \delta^{3/4}\}\bigr]\right| \leq O(\delta^{1/4}) \times O(\delta) = o(\delta)\,.
\end{equation*}
The above inequality follows from writing the expectation as sum of differences and applying the triangle inequality. Similarly, on the complementary event, we have
\begin{align*}
    \left|\mathbb E\bigl[(-1)^{\sigma+\tau}T\bigl(b-a+(\tau-\sigma)\times 4\delta e_1\bigr) \II\{|\langle a-b, e_2\rangle| > \delta^{3/4}\}\bigr]\right| \leq O(\delta^{-3/4} )\times O(\delta^2) = o(\delta)
\end{align*}
This shows that $Q_H(f_A,f_A) = o(\delta)$ and $Q_K(f_A,f_A)  = o(\delta)$. It remains to bound $Q_K(f_A,f_B)$.

\medskip

 For $a\in A$ and $b\in B$, the horizontal difference $\langle b-a,e_1\rangle$ falls in the interval $[1+5\delta, 1+7\delta]$. This means that, unless $\sigma=+1$ and $\tau = -1$, 
\[
 \langle b-a + (\sigma - \tau) \times 4\delta e_1,e_1\rangle \geq 1+ \delta \implies K(b-a + (\sigma - \tau) \times 4\delta e_1) = 0\,.
\]
On the other hand, note that when $\sigma = +1$ and $\tau = -1$, for all sufficiently small $\delta$,
\[
    \norm{b-a-8\delta e_1}^2\le(1-\delta)^2+\delta\le(1-\delta/4)^2 \implies K(b-a-8\delta e_1)\le-\delta/4
\]
This means that 
\begin{equation}\label{eq:nonuniform-mixed-gain}
    Q_K(f_A,f_B)=\frac{|A||B|}{|R|}\mathbb E\bigl[K(b-a-8\delta e_1)\bigr]\le-\frac{\delta|A||B|}{4|R|}
    =-\frac{\delta^4}{100}\,.
\end{equation}
Combining this with the preceding bounds, we obtain
\[
\begin{aligned}
    \frac83Q_K(f,f)+4Q_H(f,f) \le-c\delta^4+o(\delta^4)<0
\end{aligned}
\]
for small constant $c < 0$ and all sufficiently small $\delta$. Fix any such  $\delta$. Equations \eqref{eq:nonuniform-ratio-expansion} and \eqref{eq:nonuniform-quadratic-expansion} then give \eqref{eq:nonuniform-periodic-improvement} for every sufficiently small $\varepsilon>0$.
\end{proof}

\begin{proof}[Proof of Theorem~\ref{thm:main_counterexample}]
Choose $\rho$ as in \Cref{prop:nonuniform-counterexample}. Then
\Cref{lem:nonuniform-periodic-transfer} gives
\[
        \kappa\le\frac{\operatorname{Cr}(\rho)}{E(\rho)^3}<\frac8{9\pi^2}\,.\qedhere
\]
\end{proof}

\section{Drawings on two parallel lines}\label{sec:two-parallel-lines}

The purpose of this section is to prove Theorem~\ref{thm:2-layer}; we show that the midrange crossing constant for the class of $2$-layer drawings is $2/3$. Recall that a \emph{two-layer drawing} is a good straight-line drawing $\mathcal D=(G,p)$ of a bipartite graph $G=(A\cup B,E)$ in which the vertices of $A$ lie on a line $\ell_A$ and the vertices of $B$ lie on a parallel line $\ell_B$. 

Unlike in the uniform planar and spherical settings, no additional hypothesis on the distribution of the vertices will be required here. In fact, whether two edges cross depends only on the orders of their endpoints along $\ell_A$ and $\ell_B$, so we can assume without loss of generality that the vertices are uniformly distributed on two parallel segments.

The proof of the lower bound in Theorem~\ref{thm:2-layer} again follows an analytic approach. We first work in the continuous setting, where crossing edges are parametrized by their intersection point and the slopes of the two edges. At each point between $\ell_A$ and $\ell_B$, the crossing density can expressed in terms of differences of slopes of edges. We then apply a new local rearrangement inequality, which plays the same role as Busemann’s inequality in the earlier argument. Integrating this estimate and applying Jensen’s inequality gives us the global lower bound.

To transfer the continuous bound to finite drawings, we replace each vertex by an interval on its supporting line. This introduces additional crossings, we control this error by splitting the vertices of large degree in a way that preserves the crossing count while increasing the number of vertices by only a negligible proportion.

\subsection{The continuous inequality}

Let $S_i:=[0,1]\times\{i\}$ for $i\in\{0,1\}$, and let $\lambda_i$ be the uniform probability measure on $S_i$. Fix $\alpha\in(0,1)$, put $\beta=1-\alpha$, and equip $S=S_0\cup S_1$ with the probability measure $\mu:=\alpha\lambda_0+\beta\lambda_1$. A continuous two-layer drawing is encoded by a measurable symmetric function $w:S\times S\to[0,1]$ which vanishes on $S_i\times S_i$. Using the same normalization as in Section~\ref{sec:main}, define
\[\begin{aligned}e_{\parallel}(w)&:=\int_{S^2}w(z_1,z_2)\,\mu(\dd z_1)\mu(\dd z_2)\,,\\ \operatorname{Cr}_{\parallel}(w)&:=\int_{S^4}w(z_1,z_2)w(z_3,z_4)\II{\{[z_1z_2]\text{ crosses }[z_3z_4]\}}\prod_{i=1}^4\mu(\dd z_i)\,.\end{aligned}\]

\begin{proposition}\label{thm:two-line-continuous}
Every continuous two-layer drawing satisfies
\begin{equation}\label{eq:two-line-continuous-bound}
\operatorname{Cr}_{\parallel}(w)\geq\frac{e_{\parallel}(w)^3}{6\alpha\beta}\geq\frac23e_{\parallel}(w)^3.
\end{equation}
The constant $2/3$ is asymptotically sharp as $e_{\parallel}(w)\to0$.
\end{proposition}

We first derive the local inequality used in the proof.

\begin{lemma}\label{lem:two-line-bounded-density}
If $f:\mathbb R\to[0,1]$ is measurable and integrable, then
\[
\frac12\int_{\mathbb R}\int_{\mathbb R}|s-s'|f(s)f(s')\,\dd s\,\dd s'\geq \frac{1}{6} \left( \int_{\mathbb R} f(s) \, \dd s\right )^3\,.
\]
Equality holds when $f$ is the indicator of an interval.
\end{lemma}

\begin{proof}
Put $F(r):=\int_{-\infty}^r f(s)\,\dd s$ and let $M = \int_{\mathbb R} f(s) \, \dd s$. Using Tonelli's theorem, symmetry in $s,s'$, and the identity
\[|s-s'|=\int_{\mathbb R}\II{\{\min(s,s')<r<\max(s,s')\}}\,\dd r\,,\]
we obtain
\[
\begin{aligned}
\frac12\int_{\mathbb R}\int_{\mathbb R}|s-s'|f(s)f(s')\,\dd s\dd s'
&=\int_{\mathbb R}\left(\int_{-\infty}^r f(s)\,\dd s\right)\left(\int_r^\infty f(s')\,\dd s'\right)\,\dd r\\
&=\int_{\mathbb R}F(r)(M-F(r))\,\dd r\\
&\geq\int_{\mathbb R}F(r)(M-F(r))f(r)\,\dd r \\
&= \left[\frac{M}{2} F(r)^2 - \frac{1}{3}F(r)^3\right]_{-\infty}^\infty
=\frac{M^3}{6}\,.
\end{aligned}
\]
Here, the inequality uses $f\leq1$, and the last equality follows from the fact $F'=f$ almost everywhere. For an interval indicator, the inequality is an equality.
\end{proof}

\begin{proof}[Proof of Theorem~\ref{thm:two-line-continuous}]
Write $h(u,v)=w((u,0),(v,1))$ for $(u,v)\in[0,1]^2$, and extend $h$ by zero to $\mathbb R^2$. Set
\[\begin{aligned}\mathcal E(h)&:=\int_{[0,1]^2}h(u,v)\,\dd u\,\dd v\,,\\ \mathcal{C}(h)&:=\frac12\int_{[0,1]^4}h(u,v)h(u',v')\II{\{(u-u')(v-v')<0\}}\,\dd u \,\dd v \, \dd u' \,\dd v'\,.\end{aligned}\]
Thus $\mathcal C(h)$ counts unordered crossing pairs of edges oriented from $S_0$ to $S_1$. The normalizations are related by
\begin{equation}\label{eq:two-line-normalizations}
e_{\parallel}(w)=2\alpha\beta\mathcal{E}(h),\qquad \operatorname{Cr}_{\parallel}(w)=8\alpha^2\beta^2\mathcal C(h)\,.
\end{equation}

The key idea in the proof consists of bounding separately the amount of crossings occurring on each vertical line with horizontal coordinate between $0$ and $1$; we will use Lemma~\ref{lem:two-line-bounded-density} and Jensen's inequality to analyze each of these lines, and then we will integrate over the horizontal coordinate. 

For $t\in(0,1)$ and $x,s\in\mathbb R$, the edge from $(u,0)$ to $(v,1)$ passes through $(x,t)$ with total horizontal displacement $s=v-u$ precisely when $u=x-ts$ and $v=x+(1-t)s$; define
\[f_t(x,s):=h(x-ts,x+(1-t)s)\,,\qquad M_t(x):=\int_{\mathbb R}f_t(x,s)\,\dd s\,.\]
 The map $(x,s)\mapsto(u,v)$ has Jacobian determinant $1$. Moreover, $f_t(x,s)$ vanishes unless $x\in[0,1]$, since $x=(1-t)u+tv$. Consequently,
\begin{equation}\label{eq:two-line-slice-mass}
0\leq f_t\leq1\,,\qquad \int_0^1M_t(x)\,\dd x=\mathcal{E}(h)\,.
\end{equation}
Every ordered pair of crossing edges, outside of a null set, has a unique representation as
\[(u,v,u',v')=(x-ts,x+(1-t)s,x-ts',x+(1-t)s')\,,\qquad 0<t<1\,,\]
and the absolute Jacobian determinant of this change of variables is $|s-s'|$. To see this, we differentiate with respect to $(x,t,s,s')$, obtaining the Jacobian matrix
\[J=\begin{pmatrix}
1&-s&-t&0\\
1&-s&1-t&0\\
1&-s'&0&-t\\
1&-s'&0&1-t
\end{pmatrix}.\]
Subtracting the first row from the second and the third row from the fourth leaves the determinant unchanged and turns these rows into $(0,0,1,0)$ and $(0,0,0,1)$, respectively. It is then straightforward to see that
\[\det J=-\det\begin{pmatrix}1&-s\\1&-s'\end{pmatrix}=s'-s\,.\]

Now, we have that
\begin{equation}\label{eq:two-line-crossing-identity}
\mathcal C(h)=\frac12\int_0^1\left[\int_0^1\int_{\mathbb R}\int_{\mathbb R}|s-s'|f_t(x,s)f_t(x,s')\,\dd s\,\dd s'\,\dd x\right]\dd t\,.
\end{equation}
Applying Lemma~\ref{lem:two-line-bounded-density} at each $(t,x)$ and then Jensen's inequality on $[0,1]$, we obtain
\begin{equation}\label{eq:two-line-kernel-bound}
\mathcal C(h)\geq\frac16\int_0^1\int_0^1M_t(x)^3\,\dd x\,\dd t\geq\frac16\int_0^1\left(\int_0^1M_t(x)\,\dd x\right)^3\dd t=\frac{\mathcal{E}(h)^3}{6}\,.
\end{equation}
Here, we have also used \eqref{eq:two-line-slice-mass}.
Combining this with \eqref{eq:two-line-normalizations} and $\alpha\beta\leq1/4$ proves \eqref{eq:two-line-continuous-bound}.

For asymptotic sharpness, take $\alpha=\beta=1/2$ and $h_\delta(u,v)=\II{\{0\leq v-u\leq\delta\}}$, where $0<\delta<1$. We have $\mathcal{E}(h_\delta)=\delta-\delta^2/2$. For fixed $t\in(0,1)$ and $x\in[0,1]$, the condition $h_\delta(x-ts,x+(1-t)s)=1$ is equivalent to
\[0\leq s\leq\delta,\qquad x-ts\geq0\,,\qquad x+(1-t)s\leq1\,.\]
Thus the function $s\mapsto h_\delta(x-ts,x+(1-t)s)$ is the indicator of the interval $[0,M_t(x)]$, where
\[M_t(x):=\min\left\{\delta,\frac{x}{t},\frac{1-x}{1-t}\right\}\,.\]
A direct integration gives
\[\int_0^1M_t(x)^3\,\dd x=\delta^3-\frac34\delta^4\,.\]
Since indicators of intervals attain equality in the local inequality, \eqref{eq:two-line-crossing-identity} gives $\mathcal C(h_\delta)=\delta^3/6-\delta^4/8$. Therefore, the associated symmetric kernels $w_\delta$ satisfy
\[\frac{\operatorname{Cr}_{\parallel}(w_\delta)}{e_{\parallel}(w_\delta)^3}=\frac{4\mathcal C(h_\delta)}{\mathcal{E}(h_\delta)^3}\longrightarrow\frac23\,.\qedhere\]
\end{proof}

\subsection{Transfer to finite drawings}

In order to reduce the study of discrete two-layer drawings to the setting considered above, we shall replace each vertex by an interval on the corresponding line. However, replacing edges by intervals introduces crossings between edges that originally shared an endpoint. We first compute this error exactly, and then control it by splitting vertices of large degree.

\begin{theorem}\label{lem:two-line-finite-error}
Let $\mathcal D$ be a two-layer drawing of $G=(A\cup B,E)$, with $|A|=a\geq1$, $|B|=b\geq1$, and $m=|E|$. Then
\[
\operatorname{cr}(\mathcal D)\geq\frac{m^3}{6ab}-\frac14\sum_{v\in A\cup B}\deg(v)^2+\frac m4\,.
\]
\end{theorem}

\begin{proof}
List $A=\{a_1,\ldots,a_a\}$ and $B=\{b_1,\ldots,b_b\}$ in their orders along the two lines, oriented in the same direction. Partition $[0,1]$ into consecutive intervals $I_1,\ldots,I_a$ of length $1/a$, and also into consecutive intervals $J_1,\ldots,J_b$ of length $1/b$. Define $h$ to be $1$ on $I_i\times J_j$ if $a_ib_j\in E$, and $0$ otherwise. Then $\mathcal{E}(h)=m/(ab)$, and
\begin{equation}\label{eq:two-line-step-crossings}
a^2b^2\mathcal C(h)=\operatorname{cr}(\mathcal D)+\frac12\sum_{v\in A\cup B}\binom{\deg(v)}2+\frac m4=\operatorname{cr}(\mathcal D)+\frac14\sum_{v\in A\cup B}\deg(v)^2-\frac m4\,.
\end{equation}
Indeed, pairs of edges with no common endpoint cross exactly when their two endpoint orders are reversed, so their contribution is the original crossing count. For two distinct edges sharing an endpoint, two independent samples from the interval representing that endpoint are in either order with probability $1/2$, giving the second term. Finally, when both sampled edges come from the same cell $I_i\times J_j$, their endpoints have opposite orders with probability $1/2$; the factor of $1/2$ in the definition of $\mathcal C(h)$ gives the contribution of $1/4$ per edge. Applying \eqref{eq:two-line-kernel-bound} to \eqref{eq:two-line-step-crossings} proves the lemma.
\end{proof}

Next, we must handle vertices of large degree; this is achieved through a vertex splitting procedure. This is certainly not a novel idea, and it was already utilized in the proof of Theorem~\ref{thm:midrange_exists} in~\cite{PachSpencerToth2000}.

\begin{lemma}\label{lem:two-line-degree-splitting}
For every two-layer drawing $\mathcal D$ with $n$ vertices and $m$ edges, and every integer $K\geq1$, there is a two-layer drawing $\mathcal D'$ of a simple graph with $m$ edges, maximum degree at most $K$, and at most $n+2m/K$ vertices, such that $\operatorname{cr}(\mathcal D')=\operatorname{cr}(\mathcal D)$.
\end{lemma}

\begin{proof}
For each vertex $a\in A$ of positive degree, partition its neighbors into consecutive blocks of at most $K$ vertices in their order along the other line. We replace $a$ by $\lceil\deg(a)/K\rceil$ consecutive vertices on its supporting line, assigning one block to each new vertex in the same order as the blocks (see Figure~\ref{fig:two-line-degree-splitting}). Isolated vertices are left unchanged. Edges originally incident to different vertices of $A$ retain their relative endpoint orders. Two edges formerly incident to the same vertex of $A$ either still share an endpoint or now have their endpoints in the same order on both lines. Thus, no crossing is created or removed by the splitting procedure. The number of added vertices is at most $\sum_{a\in A}\deg(a)/K=m/K$.

Next, we perform the same operation on $B$, using the resulting order on the first line. This adds at most another $m/K$ vertices and again preserves every crossing. Degrees on the first line do not change in this second step, so the final maximum degree is at most $K$. The graph remains simple and has $m$ edges. A generic choice of vertex positions within the specified orders ensures that no three edges meet at a common interior point.
\end{proof}

\begin{figure}[tb]
\centering

\begin{tikzpicture}[x=.9cm,y=1cm,font=\scriptsize,
  layer/.style={draw=black!12,line width=.25pt},
  background edge/.style={draw=black!14,line width=.35pt},
  focal edge/.style={line width=.65pt},
  before edge/.style={focal edge,draw=red!80!black},
  after edge/.style={focal edge,draw=blue!80!black},
  vertex/.style={circle,fill=black,inner sep=0pt,minimum size=3pt},
  background vertex/.style={vertex,fill=black!24,minimum size=2.5pt}]

\foreach \panel in {0,4.9}{
  \begin{scope}[xshift=\panel cm]
    \draw[layer] (-.05,0) -- (4.6,0);
    \draw[layer] (-.05,1.8) -- (4.6,1.8);
    \foreach \u/\v in {
      .08/.06,.08/.53,
      .38/.31,.38/1.25,.38/2.13,
      .7/.53,.7/.98,
      1.02/.31,1.02/1.54,1.02/2.45,
      1.34/.98,1.34/1.85,1.34/2.76,
      2.67/1.25,2.67/2.13,2.67/3.39,
      3.03/1.54,3.03/2.76,3.03/3.66,
      3.4/1.25,3.4/3.39,3.4/3.92,
      3.8/2.45,3.8/3.66,3.8/4.21,
      4.16/3.1,4.16/3.92,
      4.5/3.39,4.5/4.49}{
      \draw[background edge] (\u,0) -- (\v,1.8);
    }
    \foreach \x in {.08,.38,.7,1.02,1.34,2.67,3.03,3.4,3.8,4.16,4.5}{
      \node[background vertex] at (\x,0) {};
    }
    \foreach \x in {.06,.31,.53,.98,1.54,2.13,2.76,3.39,3.66,3.92,4.21,4.49}{
      \node[background vertex] at (\x,1.8) {};
    }
  \end{scope}
}

\foreach \x in {.75,1.25,1.85,2.45,3.1}{
  \draw[before edge] (1.95,0) -- (\x,1.8);
  \node[vertex] at (\x,1.8) {};
}
\node[vertex] at (1.95,0) {};

\begin{scope}[xshift=4.9cm]
  \foreach \x in {.75,1.25}{\draw[after edge] (1.73,0) -- (\x,1.8);}
  \foreach \x in {1.85,2.45}{\draw[after edge] (1.95,0) -- (\x,1.8);}
  \draw[after edge] (2.17,0) -- (3.1,1.8);
  \foreach \x in {.75,1.25,1.85,2.45,3.1}{\node[vertex] at (\x,1.8) {};}
  \foreach \i/\x in {1/1.73,2/1.95,3/2.17}{
    \node[vertex] at (\x,0) {};
  }
\end{scope}

\draw[->,>=stealth,line width=.5pt] (4.73,.9)
  -- (5.22,.9);
\end{tikzpicture}

\caption{Splitting a vertex into multiple nearby vertices with bounded degrees preserves the crossings.}
\label{fig:two-line-degree-splitting}
\end{figure}

\begin{proof}[Proof of the lower bound in Theorem~\ref{thm:2-layer}]
Let $\mathcal D$ be a two-layer drawing of a bipartite graph with $n$ vertices and $m$ edges, and apply Lemma~\ref{lem:two-line-degree-splitting} with an integer $K\geq1$. Write $n'$ for the number of vertices of the resulting drawing. Since the product of its part sizes is at most $(n')^2/4$ and its sum of squared degrees is at most $2mK$, Theorem~\ref{lem:two-line-finite-error} yields the quantitative bound
\[
\operatorname{cr}(\mathcal D)\geq\frac{2m^3}{3(n+2m/K)^2}-\frac{mK}{2}\,.
\]
Put $r=m/n$ and take $K=\lceil r^{3/2}\rceil$. If $r\to\infty$, then $2m/(Kn)=O(r^{-1/2})$ and $Kn^2/m^2=O(r^{-1/2})$. Therefore,
\[\operatorname{cr}(\mathcal D)\geq\left(\frac23-O\left(\sqrt{\frac nm}\right)\right)\frac{m^3}{n^2}\,.\]
In particular, the required lower bound holds along every midrange sequence. In fact, this part of the argument requires only that $n\ll m$.
\end{proof}

\subsection{An asymptotically optimal construction}

For completeness, we provide also a construction matching the upper bound in Theorem~\ref{thm:2-layer}; again, we remark that this construction had already appeared in~\cite{2-layer}.

\begin{proof}[Proof of the upper bound in Theorem~\ref{thm:2-layer}]
Consider a pair $(n,m)$ along a sequence such that $n\ll m\ll n^2$. Let $N=\lfloor n/2\rfloor$. Place vertices $a_1,\ldots,a_N$ and $b_1,\ldots,b_N$ on the two parallel lines in increasing order, and, for an integer $d\in\{1,\ldots,N\}$, include the edge $a_ib_j$ exactly when $0\leq j-i<d$. This gives
\[M_d=Nd-\binom d2\]
edges. Consider two edges $a_i b_{i+r}$ and $a_{i+h}b_{i+h+r'}$, where $h\geq1$ and $0\leq r,r'<d$. Since $a_i$ lies to the left of $a_{i+h}$, these edges cross precisely when $b_{i+r}$ lies to the right of $b_{i+h+r'}$, that is, when $i+r>i+h+r'$. For a given $a_i$ and $a_{i+h}$, there are at most $\binom{d-h}{2}$ such pairs of crossing edges when $1\leq h\leq d-2$, and none for larger $h$. Hence the number of crossings is at most
\[\sum_{h=1}^{d-2}(N-h)\binom{d-h}{2}\leq N\sum_{h=1}^{d-2}\binom{d-h}{2}=N\binom d3.\]

Choose the least $d\in\{1,\ldots,N\}$ such that $M_d\geq m$, which exists for all sufficiently large $n$ since $m\ll n^2$. Delete $M_d-m$ edges and, if $n$ is odd, add one isolated vertex on either supporting line. These operations do not increase the number of crossings. Since $M_d-M_{d-1}\leq N$, we have $0\leq M_d-m<N$, with the harmless convention that $M_0=0$. Since $M_d\geq Nd/2$ and $m\ll n^2$, we know also that $d/N$ must approach $0$ as $n$ grows. On the other hand, the assumption that $n\ll m$ and the inequality $M_d\leq Nd$ imply that $d$ approaches infinity as $n$ grows. It follows that $M_d\sim m\sim Nd$. Putting things together, the resulting drawing satisfies
\[\operatorname{cr}(\mathcal D)\leq N\binom d3=\left(\frac16+o(1)\right)\frac{m^3}{N^2}=\left(\frac23+o(1)\right)\frac{m^3}{n^2}\,.\]
Together with the lower bound, this proves Theorem~\ref{thm:2-layer}.
\end{proof}

\bibliographystyle{plain}
\bibliography{references}

\end{document}